\documentclass[11pt,leqno]{article}

\usepackage[all]{xy}
\usepackage{amssymb,amsmath,amsthm,    url,rotating}
\usepackage{mathrsfs}
  \usepackage{graphicx,epsfig}
  \usepackage{xcolor,graphicx}
    \usepackage{makeidx}
\newcommand{\n}{\noindent}

\newcommand{\vp}{\varepsilon}
\newcommand{\bb}[1]{\mathbb{#1}}
\newcommand{\cl}[1]{\mathcal{#1}}

\newcommand{\ovl}{\overline}

\theoremstyle{plain}
\newtheorem{thm}{Theorem}[section]
\newtheorem{lem}[thm]{Lemma}

\newtheorem{pro}[thm]{Proposition}

\newtheorem{cor}[thm]{Corollary}

\theoremstyle{definition}

\newtheorem{dfn}[thm]{Definition}

\theoremstyle{remark}
\newtheorem{rem}[thm]{Remark}

\numberwithin{equation}{section}

\def\tilde{\widetilde}

\renewcommand{\tilde}{\widetilde}

 \def\R{\bb  R}
\def\RR{\bb  R}
 
\def\CC{\bb  C}
\def\KK{\bb  K}
\def\E{\bb  E}
\def\F{\bb  F}
\def\P{\bb  P}
\def\T{\bb  T}

\def\d{\delta}
\def\NN{\bb  N}
\def\N{\bb  N}
\def\RR{\bb  R}
\def\Q{\bb  Q}
\def\PP{\bb  P}
\def\CC{\bb  C}
\def\KK{\bb  K}
\def\RR{\bb  R}
\def\CC{\bb  C}
\def\KK{\bb  K}
\def\E{\bb  E}
\def\F{\bb  F}
\def\P{\bb  P}
\def\T{\bb  T}

\def\d{\delta}
\def\NN{\bb  N}
\def\ZZ{\bb  Z}
\def\Z{\bb  Z}
\def\RR{\bb  R}
\def\PP{\bb  P}
\def\CC{\bb  C}
\def\KK{\bb  K}
\def\FF{\bb  F}

\def\ov{\overline}
\def\phi{\varphi}
\def\ie{{\it  i.e.\  }}
\def\eg{{\it  e.g.}}
\def\n{\noindent}
\def\nl{\nolimits}
\def\tr{\rm  tr}
\def\a{\alpha}
\def\t{\theta}

\def\C{\mathscr{C}}
\def\B{\mathscr{B}}
\def\I{\cl  I}
\def\e{\cl  E}

\begin{document}

\title{Seemingly injective von Neumann algebras
    }

\author{by\\
Gilles  Pisier  \\
Texas  A\&M  University\\
College  Station,  TX  77843,  U.  S.  A.}

\def\C{\mathscr{C}}
\def\B{\mathscr{B}}
\def\I{\cl  I}
\def\e{\cl  E}
\def\G{\bb G}

\def\R{\bb R}
\def\RR{\bb R}
 
\def\CC{\bb C}
\def\KK{\bb K}
\def\E{\bb E}
\def\F{\bb F}
\def\P{\bb P}
\def\T{\bb T}

\def\d{\delta}
\def\NN{\bb N}
\def\N{\bb N}
\def\RR{\bb R}
\def\Q{\bb Q}
\def\PP{\bb P}
\def\CC{\bb C}
\def\KK{\bb K}
\def\RR{\bb R}
\def\CC{\bb C}
\def\KK{\bb K}
\def\E{\bb E}
\def\F{\bb F}
\def\P{\bb P}
\def\T{\bb T}

\def\d{\delta}
\def\NN{\bb N}
\def\ZZ{\bb Z}
\def\Z{\bb Z}
\def\RR{\bb R}
\def\PP{\bb P}
\def\CC{\bb C}
\def\KK{\bb K}
\def\FF{\bb F}

\def\v{\varphi}
\def\ov{\overline}
\def\phi{\varphi}
\def\ie{{\it i.e.\ }}
\def\eg{{\it e.g.}}
\def\n{\noindent}
\def\nl{\nolimits}
\def\tr{{\rm tr}}

 \pagenumbering{roman}
 \maketitle
 
\begin{abstract}    
 We  show that a QWEP von Neumann algebra
 has the weak* positive   approximation property
 if and only if it is seemingly injective in the following sense:
 there is a factorization of the identity of $M$
 $$Id_M=vu: M{\buildrel u\over\longrightarrow}   B(H) {\buildrel v\over\longrightarrow} M$$ 
 with $u$ normal, unital, positive and $v$   completely contractive.
 As a corollary, if $M$ has a separable predual, $M$ is isomorphic
 (as a Banach space) to $B(\ell_2)$.
 For instance this applies (rather surprisingly)
 to the von Neumann algebra of any free group.
 Nevertheless, since $B(H)$ fails
 the approximation property (due to Szankowski)  there are $M$'s (namely  $B(H)^{**}$ and certain finite examples defined using ultraproducts) that are not  seemingly injective. Moreover, for
 $M$  to be seemingly injective  it suffices to have the above factorization
 of $Id_M$ through $B(H)$ with $u,v$   positive (and $u$ still normal).
   \end{abstract}

 \thispagestyle{empty}
 
 
\vfill\eject


\setcounter{page}{1}

 \pagenumbering{arabic}
 

 A von Neumann algebra $M$ is called injective
 if there exists a Hilbert space $H$ and an isometric normal $*$-homomorphism $u: M \to B(H)$
 such that there is a projection $P: B(H) \to u(M)$ onto $u(M)$ with $\|P\|=1$, which 
 by Tomiyama's theorem (see \cite{Tak}) is automatically 
 a completely positive and completely contractive conditional expectation.
 Injective von Neumann algebras play a central role in operator algebra theory (see e.g. \cite{Tak,Tak2}).
 Following the work of Connes \cite{Co}, Choi and Effros \cite{[CE2],[CE4]} proved that
  a $C^*$-algebra $A$ is nuclear if and only if its bidual $A^{**}$
 is injective.
  Injectivity is in some sense analogous to  amenability for groups.\\
  Equivalently $M$ is injective if and only if the identity $Id_M$ admits, for some $H$, a factorization of the form
    \begin{equation}\label{ee1}
      Id_M: M {\buildrel u\over\longrightarrow} B(H) {\buildrel v\over\longrightarrow} M\end{equation} 
 where
 $u: M\to B(H)$ and $v: B(H) \to M$ are completely contractive maps.
 In fact, it suffices for injectivity to require $u,v$ completely bounded  (see Remark \ref{26/10}).
 The fundamental injective example   is of course $B(H)$ itself.
  The main non-injective example   is the von Neumann algebra $L(\F)$ of a free group $\F$
  with at least two generators.  Actually, by known results (see \cite{[HP2]})  
  for any embedding $L(\F) \subset B(H)$ as a von Neumann subalgebra there is no bounded
  linear projection onto $L(\F)$. Moreover when $M=L(\F)$ there is no factorization
  as above with $u,v$ both completely bounded (see Remark \ref{26/10}).
  The initial motivation of this note is the observation that when $M=L(\F)$, the identity $Id_M$ 
  still admits a factorization as above but with $u$ merely a normal unital positive (linearly) isometric embedding and $v$ still completely contractive. When this holds we say that $M$ is   ``seemingly injective".
  Actually,  for this to hold it suffices to have such a factorization with  $u$ and $v$ both   positive, and $u$  normal.\\
  In Theorem \ref{ot1}, we show that this property characterizes the $M$'s that have a certain form of  weak* positive   approximation property (in short PAP). The latter was proved for $M=L(\F)$ in \cite{CaHa} following Haagerup's well known work \cite{Hainv}. Our result can be viewed as
  analogous to the equivalence between injectivity and
  semidiscreteness (see \cite[p. 173]{Tak2}).
In \cite[p. 12]{EL}, Effros and Lance  called ``semidiscrete" the algebras $M$ that have the 
  weak* completely positive   approximation property (in short weak* CPAP).
  Their paper contains many  important fore-runners
 of Connes's later results on the equivalence of injectivity and semidiscreteness.
   Connes  \cite{Co} proved the latter equivalence  for factors
  and   Choi and Effros \cite{[CE4]} extended it to arbitrary von Neumann algebras.
   We develop further the parallelism between 
   seemingly injective and injective in \S \ref{inj}.

 \section{Main result}
 
 Throughout this paper,  we abbreviate ``approximation property" by AP.\\
 We start by  introducing the  weakening of injectivity,   
 which is the center of interest of this paper:
 \begin{dfn}\label{d1}
We will say that a  von Neumann algebra (or a dual operator space)
 $M$ is ``seemingly injective" 
 if   $Id_M$ admits, for some $H$, a factorization of the form
      \begin{equation}\label{e1}
      Id_M: M {\buildrel u\over\longrightarrow} B(H) {\buildrel v\over\longrightarrow} M \end{equation}
    with  $u$ normal, unital and positive (and hence $\|u\|=1$), and  $\|v\|_{cb}=1$.
    Note that $u$   unital  and $vu=Id$ implies that $v$ is also unital, and hence completely positive. \\
   Dropping  the unital positivity assumption on $u$, if  this factorization holds  with $u,v$  such that $\|u\|=1$ and $\|v\|_{cb}=1$
    we will say   that $M$ is remotely   injective.
    \end{dfn}
    \begin{rem}\label{or0} In the preceding situation, $u$ is isometric and normal, so by  Krein-Smulian
    the subspace $u(M)\subset B(H)$ is weak* closed, and by Sakai's predual uniqueness theorem 
    (see \cite[p. 30]{Sak}),
    $u^{-1}: u(M) \to M$ is weak* continuous.
      \end{rem}
    \begin{rem}\label{or1} If    $M$ with separable predual is  seemingly injective (or remotely injective),
    we may take $H=\ell_2$. Indeed, by the preceding remark
    $u(M)$ has a separable predual, namely $B(H)_*/Z$ where $Z$ is the preanihilator of $u(M)$.
    It follows that $u(M)$ is normed by a countable subset of the unit ball of $B(H)_*$.
    Therefore there is a separable Hilbert subspace $K\subset H$ such that the compression
    $\Psi: x\mapsto P_K x_{|K }$ is an isometric  normal unital embedding of $u(M)$ in $B(K)$.
    Repeating this argument for $M_n(u(M))$ for all $n\ge 1$ (and augmenting $K$) we can obtain a separable $K\simeq \ell_2$
    such that the preceding embedding $\Psi: u(M)\to B(K)$ is also  completely isometric.
    Then by the injectivity of $B(H)$ the embedding $u(M)\subset B(H)$
    factors as $u(M) {\buildrel  {\Psi}\over \longrightarrow } B(K) {\buildrel  {w}\over \longrightarrow } B(H)$
    with   $w$ completely contractive. Replacing $v$ in \eqref{e1} by $vw$
    this shows that we may assume that $H=\ell_2$.
     \end{rem}
       \begin{rem}\label{26/10} If $u,v$ are both assumed completely bounded
       in \eqref{ee1}  then
       $M$ is injective (and here we do not need to assume $u$ normal).
       This is due to Haagerup (see \cite[Th. 23.7]{P6} together with \cite[Cor. 22.19]{P6}), refining previous results  
     by Christensen-Sinclair and the author
      independently showing that the existence of a completely complemented embedding
      $M \subset B(H)$ (as a von Neumann subalgebra) implies injectivity.
      See   \cite[\S 23.7]{P6} for more detailed references.
        We return to  this topic  in Remark \ref{rc}.
       \end{rem}
    \begin{rem} We could also consider a variant with a constant $c$,  and say 
    seemingly $c$-injective if  $\|v\|_{cb}
\le c$. 
There is a natural notion of seemingly nuclear which we do not spell out
involving the analogue of the completely positive AP. All this seems to deserve further study.
\end{rem}
 \begin{rem}\label{rf1} Let $p$ be a projection in $M$. We may view $pMp$ as usual
 as a von Neumann algebra  with unit $p$. It is easy to check that
 if $M$ is seemingly injective, so is $pMp \oplus \CC (1-p)$.
 Moreover, if $M$ is seemingly injective, so is any von Neumann 
 (unital) subalgebra $N\subset M$ for which there is a 
unital  positive (automatically c.p.) projection $P: M\to N$ onto $N$. By the last assertion in Theorem \ref{ot1}
 ``unital positive" can be replaced by ``positive" in the preceding assertion.
 In particular $pMp$ is seemingly injective.  
\end{rem}

 \begin{rem} One major dissemblance with injectivity is that
 it is unclear (and a priori unlikely to be true) whether the tensor product
 of two seemingly injective algebras $N,M$ is seemingly injective,
 even if $N=B(K)$ with $K$ Hilbert. This is of course due to the fact
 that, unlike for c.p. maps, we cannot tensorize unital positive maps.
 Analogously, we do not know whether the commutant $M'$ 
 is seemingly injective  if $M$ is so.
\end{rem}

 \begin{rem}\label{mixed} Following \cite[p. 275]{OR2},  given a constant $c\ge 1$,
 an operator space $Z\subset B(H)$ is $c$-mixed injective  if $Z$ is the range of   a projection
 $P: B(H) \to Z$ with $\|P\|\le c$. With this terminology,  
 any remotely injective von Neumann algebra $M$ is  isometric
 to a weak* closed subspace $Z\subset B(H)$ that is $1$-mixed injective.
 Indeed, if \eqref{e1} holds we can take $Z=u(M)$ and $P=uv$.
 \end{rem}
We will use the weak expectation property (in short WEP)
originally introduced by Lance.
Recall that a $C^*$-algebra $A$ is called WEP
(resp. QWEP)
if the inclusion $i_A: A \to A^{**}$ into its bidual
factors via completely positive 
(in short c.p.) and completely contractive (in short c.c.)
maps through some $B(H)$
 (resp. if it is a quotient
of a WEP $C^*$-algebra). 

\begin{rem}\label{kiki} Kirchberg proved that a von Neumann algebra $M$ is
QWEP if and only if  there is (for some $H$) a u.c.p. map $v: B(H) \to M$
that is a metric surjection
(i.e. $v$ maps the open unit ball of $B(H)$ onto 
the open unit ball of $M$). In particular, this shows  the following.\end{rem}

\begin{pro}\label{kiki1} A von Neumann algebra $M$ is seemingly injective
if and only if  there is (for some $H$) a u.c.p. metric surjection $v: B(H) \to M$
that admits a normal unital positive lifting $u: M \to B(H)$.
\end{pro}

 Our main result is as follows (see \S \ref{bg} for the definitions
 on the various forms of AP):
 
\begin{thm}\label{ot1}
 The following properties of a von Neumann algebra  $M$ are equivalent:
 \item{\rm (i)} $M$  is seemingly injective,
   \item{\rm (ii)} There is a net of integers $n(\a)$ and normal  finite rank
   maps $T_\a: M\to M$ of the form
   $$    M {\buildrel u_\a\over\longrightarrow} M_{n(\a)} {\buildrel v_\a\over\longrightarrow} M$$
   such that $u_\a,v_\a$  are both unital and positive
    (so that $\|u_\a\|\le 1$,  $\|v_\a\|\le 1$), $u_\a$ is normal
   and $T_\a(x)=v_\a u_\a(x) \to x$ weak* for any $x\in M$ (we then say that
$M$ has the matricial weak* PAP).
     \item{\rm (iii)} $M$
 is QWEP and
 there is a net $(T_\a)$ of normal  finite rank unital positive maps on $M$
such that $T_\a(x)  \to x$ weak* for any $x\in M$ (in which case we say that
$M$ has the weak* PAP).\\
Moreover, for $M$ to be seemingly injective it suffices that
 there is a factorization
of the form \eqref{e1}
with $u$ normal and $u,v$ both   positive. 
\end{thm}

 \begin{rem} We do not know whether  
 any $M$ with the weak* PAP
 is QWEP, although the c.p. analogue holds.
 See \cite{JLM} for an illuminating discussion on the latter assertion.
 \end{rem}
 \begin{rem}[On free groups]\label{caha} If $G$ is a free group
 the metric AP  for 
 $C^*_\lambda(G)$ was discovered by Haagerup \cite{Hainv},
 as well as the existence of a sequence   of 
 completely contractive finite rank
 multipliers on
 the von Neumann algebra $L(G)=C^*_\lambda(G)''$ that form a weak* approximation of 
 the identity. 
 Of course the    \emph{completely  positive} analogue is excluded since
 it implies injectivity. Nevertheless,
 using very similar ideas,
  Haagerup and de Canni\`ere
\cite[Th. 4.6]{CaHa} proved that there is a sequence  of finite rank
 normal  multipliers   that are \emph{positive}, unital
 and form a weak* approximation of the identity on $L(G)$
 (see also \cite{JVa} for groups with property RD with respect to
 a conditionally negative definite length function).
 In fact they obtain, for any integer $k\ge 1$ fixed in advance,
 a sequence of unital $k$-positive multipliers.
 By (iii) $\Rightarrow$ (i) this shows that 
 $L(G)$     is seemingly injective,  and  actually
 we can find a factorization with a $u$ that is  $k$-positive, unital and normal.
 \end{rem}
 \begin{rem}[On Jordan algebras]  The ranges of unital positive projections $P$ on $B(H)$
 (such as $P=uv$ which projects onto $u(M)$) are described
 by Effros and St\o rmer in \cite{ESt}.
 They show that such ranges are Jordan algebras
 for the Jordan product defined by 
 $a\circ b: =P(ab+ba)/2$.
  In the situation of Definition \ref{d1} the map
 $u: M \to P(B(H))$ being isometric is a Jordan morphism,
 and hence satisfies $P(u(x)u(y)+u(y)u(x))=u(xy+yx)$
 for all $x,y\in M$.

Perhaps, some version of Theorem \ref{ot1} remains valid 
when $M$ is only a unital JBW* algebra, for instance: 
if such an $M$ is seemingly injective, 
does it have the weak* metric AP ? \\ 
A variant of this question is as follows. Consider a 
unital weak* closed subspace $Z\subset B(H)$ which is the range of a unital contractive projection.
Does $Z$ have the weak* metric AP ?
Perhaps one should consider the same question with $B(H)$ replaced
by a unital JBW* algebra  satisfying some suitable variant of the weak* positive metric  AP ?

 The range of a contractive projection $P$ on $B(H)$
 (or on a $J^*$-algebra) is described by Friedman and Russo
 in \cite{FR}. 
 They show that such a range  is a
 Jordan triple system in the triple product
 $(x,y,z) \mapsto P( xy^*z+  zy^*x)/2$.
 The latter can be faithfully represented as a 
 $J^*$-algebra in the sense of Harris \cite{Harr,Harr2}.
 A $J^*$-algebra is a norm closed subspace 
 $X\subset B(H,K)$ ($H,K$ Hilbert spaces) that is closed for the operation
 $X \ni x\mapsto xx^*x$. 
 If $u$ and $v$ are contractive in \eqref{e1}, $u: M \to P(B(H))$ is a triple morphism
which means that
$u ( xy^*z+zy^*x  ) =P( u(x)u(y)^*u(z)+ u(z)u(y)^*u(x)) )$
for any $x,y,z \in M$.
 \end{rem}
  \begin{rem}\label{cou3}
  Let $M$ be seemingly injective as in Definition
  \ref{d1} with $u$ normal unital positive
  and $v$ unital c.p. Then it is easy to check that 
  for any unitary $U$ in $M$, since $\|u(U)\|=1$ and $v(u(U))=U$,  $u(U)$ must be in the multiplicative domain of 
  $v$. 
See e.g. \cite[\S 5.1]{P6} or \cite{Pa2} for background on multiplicative domains.   
It follows that $u(M)$  is included in the multiplicative domain of 
  $v$, and hence if $P=uv$, we have 
  $$ 
  \forall x,y\in M\quad u(xy)=P(u(x)u(y))\text{ and }  u(x)^*=u(x^*) .$$
Thus if we equip the range $P(B(H))$ ($=u(M)$) with the product
defined by 
 $a\circ b: =P(ab)$ as in \cite[Th. 3.1]{[CE3]}, we find 
 a copy of $M$.
\end{rem}

 \begin{rem}\label{co}
 The equivalence (i) $\Leftrightarrow$ (ii) in Theorem \ref{ot1} is analogous to the
 equivalence of injectivity
 and  the weak* CPAP
 (also often called ``semidiscreteness") for von Neumann algebras. 
 This is a celebrated result of Connes \cite{Co} and
 Choi-Effros \cite{[CE4]}  (see also \cite{Wac} for an alternate proof).
As   mentioned by Connes \cite[p. 104]{Co} part of
his argument  for injective $\Rightarrow$ semidiscrete is implicit in Effros and Lance's \cite[proof of Prop. 4.5]{EL}. More details on that kind of   argument (that we also use below) can be found 
at least in the semi-finite case in   \cite[Th. 8.12, p. 166]{P6} (where unfortunately
the reduction to the semi-finite case is incorrect).
 \end{rem}
  \begin{rem}[About ``hyperfiniteness"]\label{hyp}
  By Connes's results \cite{Co}, if $M$ is a  finite injective von Neumann algebra
  there is a net of  \emph{completely positive} normal unital finite rank \emph{projections} (i.e. idempotent maps) $T_\a : M \to M$ that tend pointwise weak* to the identity of $M$.
  One could wonder whether in the seemingly injective case
  the same holds with  \emph{completely positive} replaced by \emph{positive}.   \end{rem}
  
 It is known (see \cite{CS}) that  if we restrict
 to the non-nuclear and  separable predual case, all injective von Neumann algebras
 are isomorphic as  Banach spaces. Actually any
 infinite dimensional injective operator system on a separable Hilbert space  
 is isomorphic either to $\ell_\infty$ or to $B(\ell_2)$
 (see \cite{RW},  and also \cite{Blow}  for a related result).
  Curiously, 
 the same is true for  the free group factors as in Remark \ref{caha}, because of the following
 fact well known to specialists.
 
 \begin{pro}\label{op1} If we restrict
 to the non-nuclear 
  case, any von Neumann algebra $M$  that is isomorphic (as a Banach space)
to a complemented subspace of $B(\ell_2)$
 is isomorphic (as a Banach space) to $B(\ell_2)$. 
 \end{pro}
   \def\BB{\bb B}
  \begin{proof}
   This follows by a well known application of Pe\l czy\'nski's decomposition method.
    Let $\B=B(\ell_2)$ and
   let  $\BB$ denote the $\ell_\infty$-sense direct sum of the family $\{M_n\mid n\ge 1\}$.
    On one hand by elementary arguments one shows easily that  $\B$ embeds completely 
   isometrically  as a complemented subspace in $\BB$.
    On the other hand, by a result due to S. Wassermann
  (\cite{Wa} or \cite[Th. 12.29]{P6}) $\BB$ is isometric to a complemented subspace of any non-nuclear $M$.
  In both cases the projections are contractive, but this is irrelevant at this point.
  All we need is the following
  \\
 {\bf Fact (Decomposition method)}: Let $B,M $  be Banach spaces
 such that each is isomorphic to a complemented subspace of the other.
 If we also assume 
 $B \simeq  \ell_\infty(B)$ then $M \simeq B$.
    
    Applying this with $B= \ell_\infty(\B)$ and $M=\B$, we find $\B\simeq \ell_\infty(\B)$.
 Then we can apply this again but with $B=\B$ and $M$ unchanged. This yields the proposition.
 \\
We now  include the proof of the above fact  for the reader's convenience.
Assume that $B,M,X,Y$ are Banach spaces such that
   $B\simeq M \oplus X$ and 
    $M\simeq B \oplus Y$.
    Since $B \simeq M \oplus X\simeq B \oplus Y \oplus X$, 
    the isomorphism $B \simeq  \ell_\infty(B)$
    implies $B \simeq  \ell_\infty(B \oplus Y \oplus X)$.
    By absorption (i.e. since $\N \cup\{0\}\simeq \N$), we have $$B\oplus Y \simeq \ell_\infty(B \oplus Y \oplus X) \oplus Y
    \simeq
     \ell_\infty(B \oplus Y \oplus X)\simeq B.$$
    Since $M\simeq B\oplus Y$, we conclude   $M \simeq B$.
   \end{proof}
 \begin{cor} 
 The free group factors $L(G)$ (or the factors described in  Remark \ref{caha}) are 
 isomorphic (as  Banach spaces) to $B(\ell_2)$. 
 \end{cor}
   \def\BB{\bb B}
  \begin{proof}
  By  Theorem \ref{ot1} and Remark \ref{or1},
  $L(G)$ is isometric to a complemented subspace of $\B$.
   \end{proof} 
   
   \begin{rem} The assumption in Proposition \ref{op1} holds for any QWEP  $M$ with separable predual that has the weak* BAP, see  Remark \ref{or2}.
 Indeed, if $M_*$ is separable and has the BAP there is a \emph{sequence} of
    finite rank normal maps on $M$ tending weak* to the identity.
   \end{rem}

   Our terminology and notation is mostly standard, except that 
   we   denote by $M \otimes N$ the \emph{algebraic} tensor product
   of two Banach spaces. When $M$ is a von Neumann algebra we also denote by $\ovl M$ the complex
   conjugate algebra, i.e. the ``same" algebra but with complex conjugate complex multiplication,
   so that $\ovl M$ is anti-isomorphic to $M$. For any $y\in M$ we denote by
   $\ovl y$ the same element considered as an element of $\ovl M$.
   As is well known the mapping $\ovl y \mapsto y^*$
   defines an isomorphism from $\ovl M$ onto the opposite algebra $M^{op}$ of $M$
   (i.e. the ``same" algebra but with  product in reverse order).
   
 For basic facts and undefined notions,
 we refer the reader to \cite{Tak,Tak2}  for operator algebras
 and to \cite{ER,Pa2,P4} for operator spaces.

 \section{Approximation and lifting properties}\label{bg}
 
 A Banach space $X$
has the metric AP (in short MAP) if its identity $Id_X$
is the pointwise limit of a net of finite rank contractions.
If $X$ is a dual space we say that it has 
the weak* MAP if 
there is such a net but converging 
pointwise to $Id_X$ for the weak* topology. 
Equivalently, assuming $X=(X_*)^*$, this means that $X_*$ has the MAP.
Indeed, by local reflexivity the finite rank contractions on $X$ may be assumed
weak*  continuous and,
  taking convex combinations,
a pointwise weakly convergent  approximating net of finite rank maps on  $X_*$
can be transformed into a norm convergent one.

An operator system $X\subset B(H)$
has the positive MAP (in short PMAP)  if its identity $Id_X$
is the pointwise limit of a net of finite rank   positive contractive maps.
If $X$ is weak* closed in $B(H)$ we say that $X$ has the
weak*   PMAP  if 
there is such a net  formed of weak* continuous   maps that converge
pointwise to $Id_X$ for the weak* topology.
In the particular case when $X$ is a von Neumann algebra,
the existence of a uniformly bounded net of positive normal finite rank maps tending pointwise weak*
to the identity implies the weak*   PMAP (see Lemma \ref{upmap}).
If the net is formed of positive unital maps (as in Definition \ref{defap})
these are  contractive so the weak*   PMAP is automatic.

To emphasize the parallel with the weak* CPAP of \cite{EL}
we adopt the following definitions (it would be more precise
to add ``unital" to weak* PAP and weak* CPAP, but we choose to
abbreviate):

\begin{dfn}\label{defap} We say that a von Neumann algebra $M$ has the 
weak* PAP if there is a net of unital positive normal finite rank maps $(T_\a)$ 
that tend pointwise weak* to the identity on $M$.\\
We say that $M$ has the 
matricial weak*  PAP if 
in addition the maps $(T_\a)$ admit a matricial factorization as in
  (ii) in Theorem \ref{ot1}.\\
  We say that $M$ has the 
weak*  CPAP (resp. matricial weak*   CPAP) 
 if in addition to the preceding properties the maps
 $(T_\a)$  (resp. $u_\a,v_\a$) are all c.p. 
 \end{dfn}

\begin{rem}\label{pb}
Let $A$ be a $C^*$-algebra with a (self-adjoint 2-sided closed) ideal $I$
so that $A/I$ is a quotient $C^*$-algebra. Let $q:A \to A/I$ be the quotient map. 
We will denote by $i_A: A \to A^{**}$ the canonical inclusion.
It is well known that the bidual $A^{**}$ (which  is a von Neumann algebra)  admits
a decomposition as $I^{**} \oplus (A^{**}/I^{**})$.
In particular, there is a contractive lifting $A^{**}/I^{**} \to A^{**}$
and a fortiori from $A/I $ to $A^{**}$.
The question whether, when $A/I$ is separable, there is always
 a bounded (or even an isometric) lifting from $A/I $ to $A$ has remained
 open since the works of Andersen and Ando from the 1970's.
 In the broader setting of $M$-ideals 
counterexamples are known (see \cite{HWW}),
but not for ideals in $C^*$-algebras.
In \S \ref{sq} we propose an approach 
  to this question 
based  on the conjecture  that there are QWEP
von Neumann algebras that are not
  remotely injective.
  
 We will use  several facts due to Andersen \cite[Th. 7]{TBA}
 (in the real linear setting)
    and Ando \cite{An} (in full generality), as follows.
     If    $A/I$  has  the   MAP  and is separable there
     exists a    contractive (and hence isometric) lifting $r: A/I \to A$.
     More generally,   if $X\subset A/I$ is a separable closed
subspace with  the MAP, the
inclusion    $X\subset A/I$ admits a contractive  lifting  $r: X \to A$.
    In fact,  Ando proved that any map $T: X \to A/I$
    from a separable Banach space $X$ that is 
    the pointwise limit of a net of finite rank contractions
    from $X$ to $A/I$ admits a contractive lifting $r: X \to A$,
    so that $qr= T$.\\
    It seems to have remained open ever since Ando's paper  \cite{An} whether this holds without
    the approximability assumption on $A/I$. This problem was studied
    in Ozawa's PhD thesis \cite{Ozth} (see also \cite{Ozllp} and \cite{[O3]}). 
    \end{rem}

    If    $A/I$  is separable and has  the PMAP,
    T.B. Andersen \cite[Th. 7 (3)]{TBA} proved that  there
     exists a positive   isometric lifting $r: A/I \to A$.
Moreover (see \cite[Th. 7 (2)]{TBA}), in the unital case  if $X\subset A/I$ is a separable
operator system with  the PMAP, the
inclusion    $X\subset A/I$ admits a positive    lifting  $r: X \to A$ that is
 isometric on the   self-adjoint  elements of $X$.
 The proof consists in a reduction to the  case when $X$ is finite dimensional,
 which, incidentally,  is sketched below in Lemma \ref{Mi}.
 By an elementary argument (see Lemma \ref{oo1}),
 one can
obtain a unital positive  lifting  $r: X \to A$.
  \begin{rem}\label{arv}  The modern way to think of Ando's theorem is through
  Arveson's principle (see \cite[p. 351]{Arv}) that says that in the separable case  pointwise limits of ``nicely" liftable maps are ``nicely" liftable.
More precisely, assume given  a separable operator system
  $X$ and a net of maps $u_i : X \to A/I$ ($A/I$ being a quotient $C^*$-algebra), if each $u_i$  admits a lifting in an admissible class (to be defined below) and converges pointwise to a map $u: X \to  A/I$ then
  $u$ itself admits a lifting in the same   class.
  
  A bounded subset  $\cl F\subset  B(X,A)$ will be called admissible if 
   for any pair $f,g$ in $\cl F$
   and any $\sigma \ge 0$ in $I_+$
   the mapping
   $$x\mapsto \sigma^{1/2} f(x)  \sigma^{1/2}+ (1-\sigma)^{1/2} g(x) (1-\sigma)^{1/2} $$
   belongs to $\cl F$.
   Let  $q: A \to A/I$ denote the quotient map and let
   $$q(\cl F)=\{ qf\mid f\in \cl F\}.$$
  Then Arveson's principle 
    (see \cite[p. 351]{Arv}) says that
    for the topology of pointwise convergence
   on $X$ we have
   $$\ovl{q(\cl F) } =q(\ovl{ \cl F  } ).$$
   Actually we do not even need to assume $\cl F$ bounded
   if we restrict to the pointwise convergence on a countable
   subset of $X$.
\begin{rem}\label{rarv}
In particular if $X$ is finite dimensional, we do not need to assume $\cl F$ bounded.
For instance, taking for $\cl F$  the class of positive maps, let  $X$ be a  finite dimensional operator system. If a  map $u: X \to A/I$
  is the pointwise limit of  maps that admit positive liftings, then $u$ itself admits a positive lifting.
\end{rem}
  
  The   classes   of positive contractions,
  unital positive maps,   unital c.p. maps, 
   contractions (resp. 
  complete contractions) are all admissible and in the latter case
  $X$ can be an arbitrary Banach (resp.  operator) space.
  More generally, for each fixed $k\ge 1$, the class
  of unital $k$-positive maps on an operator system and  that of 
  $k$-contractions on an operator space (meaning maps $u$ such that 
  $Id_{M_k} \otimes u$ is contractive) are admissible.
  
  The proof uses quasi-central approximate units in $I$. By this we mean a 
  non decreasing
  net $(\sigma_\a)$
    in the unit ball of $I_+$ 
  such that   for any $a$ in $A$ and any $b$ in
${\cl I}$
\begin{equation}\label{da16}
\|a\sigma_{\a}-\sigma_{\a} a\|\to 0\quad \hbox{and}\quad
\|\sigma_{\a} b -b\| \to 0.\end{equation}
  The reasoning is formally  the same in all cases as can be checked
  in the presentations \cite[p. 266]{[Da2]}  or \cite[p. 46 and p. 425]{P4}.
  Using this principle, one can  reduce the lifting problem
  of an approximable map $u: X \to  A/I$ roughly to that
  of lifting  finite dimensional subspaces of  $  A/I$.
 \end{rem}
 
 The next  lemma is a  simple well known fact
 describing how to obtain unital   liftings  
 when the  map to be lifted is itself unital. 
\begin{lem}\label{oo1}
Let $E$ be an operator system and let $u: E \to A/I$ be a unital positive
mapping.  
 If there is a positive  lifting $r: E \to A$ then there is
 a unital positive   one
 $r': E \to A$.
\end{lem}
\begin{proof} Let $(\sigma_\a)$ be a quasi-central approximate unit in $I$.
Let $f$ be a state on $E$.
Let $r_\a(x)=f(x) \sigma_\a + (1-\sigma_\a)^{1/2}r(x) (1-\sigma_\a)^{1/2}.$ 
  Then 
  $r_\a(1) -1=(1-\sigma_\a)^{1/2}[r(1) -1] (1-\sigma_\a)^{1/2}$
  and $r_\a$ is still a positive  lifting. In particular $q (r_\a(1))=1$.
  Since $r(1) -1\in I$ we have
  $\|r_\a(1) -1\|\to 0$ by \eqref{da16}.
  Choosing and fixing $\a$ large enough we may assume that
  $r_\a(1)$ is invertible.
  We  then
  set for any $x\in E$
  $$r'(x)= r_\a(1)^{-1/2} r_\a(x) {r_\a(1)}^{-1/2}.$$
  By functional calculus $q(r_\a(1)^{-1/2})=q(r_\a(1))^{-1/2}=1$.
  Thus $r'$ is a positive unital lifting for $q: A \to A/I$.
\end{proof}

 The next Theorem \ref{Mi}  is the basic ingredient used 
in \cite{TBA} by T.B. Andersen to
 prove that   positive maps admit  positive liftings   when $E$
 is separable  with the PMAP.
\begin{thm}[Vesterstr\o m \cite{Ves}]\label{Mi} 
Let $E$ be a f.d.  operator system and let $u: E \to A/I$ be a unital positive
mapping.  
Then $u$ admits   a unital positive  lifting $r: E \to A$ (i.e. a map such that $qr(x)=u(x)$
for any $x\in E$).
\end{thm}
 
In \cite{RoSm} Robertson and Smith give 
 a quick direct proof of Theorem \ref{Mi}.
Moreover, using an averaging argument 
their proof  shows that, for any fixed $n$, 
Theorem  \ref{Mi} remains valid if ``positive" is replaced by 
``$n$-positive" (see also \cite{Kav} for a different argument).

\section{Standard forms}\label{stan}
 
We will   use 
   the fact 
   (due to Araki, Connes and Haagerup, see \cite{Hasta})
   that
any von Neumann algebra $M$ admits a ``standard form'', which means that there is a triple $(H,J, {P}^\natural )$ consisting of a Hilbert space $H$ such that $M\subset B(H)$ (as a von Neumann algebra), an anti-linear isometric involution 
$J:H\to H$ and a cone $P^\natural \subset H$ such that
\begin{enumerate}
 \item[(i)] $J MJ=M^\prime$ and $J xJ=x^* \qquad\forall x\in M\cap M'.$
 \item[(ii)] $ {P}^\natural\subset \{\xi\in H\mid J\xi=\xi\}$   
 and $ {P}^\natural$ is self-dual, i.e.
 $$ {P}^\natural =\{\xi\in H\mid  \langle\xi,\eta \rangle\geq   0\ \forall\eta\in P^\natural \}.$$
 \item[(iii)] $\forall x\in M\quad J xJ x(P^\natural )\subset P^\natural$.
 \item[(iv)] For any $\varphi$ in $M^{+}_{*}$ there is a unique $\xi_\varphi$ in $P^\natural$ such that $\varphi(x)=\langle\xi_\varphi ,x\xi_\varphi \rangle$ for any $x$ in $M$.
\end{enumerate}
\begin{rem}\label{rebla} Let $\tau$ be a normal faithful normalized trace on $M$.
Then the usual representation $M \to B(L_2(\tau))$ (of $M$ acting by left multiplication)
realizes $M$ in standard form, with $J: L_2(\tau) \to L_2(\tau)  $ defined by $J(x)=x^*$,
and ${P}^\natural =L_2(\tau)_+$. For any $\xi\in L_2(\tau)_+$ and $y,x\in M$ we have
$J yJ x \xi=  x \xi y^*$, and if $\xi=1$ we have $\tau(xy^*)=\langle \xi, J yJ x \xi\rangle$.
\end{rem}
\begin{rem}\label{cy} Note that when $\phi$ is a faithful normal state, the 
unit vector $\xi_\varphi$ in $P^\natural$ is both separating and cyclic.
\end{rem}
\begin{rem}\label{cy2} 
For any
unit vector $\xi \in P^\natural$ 
we have for any $\sum \ovl y_j \otimes x_j \in \ovl M \otimes M$
$$|\sum \langle \xi , J y_j J x_j \xi \rangle | \le 
\|\sum \ovl y_j \otimes x_j\|
_{\ovl M \otimes_{\max} M}.$$
Indeed, $(\ovl y,x) \mapsto JyJ x$ 
extends to a (contractive) unital $*$-homomorphism 
from 
${\ovl M \otimes_{\max} M}$ to $B(H)$.
\\
In particular, when $M$ is finite with $\tau$ as in Remark \ref{rebla}
we have
$$|  \sum  \tau( y^*_j   x_j  ) | \le 
\|\sum \ovl y_j \otimes x_j\|
_{\ovl M \otimes_{\max} M}.$$
\end{rem}

\begin{dfn}
Let $A$ be a unital $C^*$-algebra.  A sesquilinear form $s:A\times A\to\mathbb{C}$
(equivalent to  a linear form on $\ovl A \otimes A$) will be called ``bipositive'' if it is both positive definite (i.e.  such that $s(x,x)\ge 0$ for all $x\in A$) and 
such that
\begin{equation}\label{ee2}
 s(a,b)\geq0\qquad\forall a,b\in A_+ .
\end{equation}
We call it normalized if $s(1,1)=1$.
\end{dfn}
To any  normalized such form $s$ we associate a state $\varphi_s$ defined by
$\varphi_s(x)=s(1,x)$.

\begin{dfn}
Let $M$ be a von Neumann algebra.
A  bipositive (sesquilinear) form $s$ on $M\times M$
such that $x\mapsto s(1,x)$ is a normal state on $M$ is called self-polar if 
  for any $\psi\in M^{*}_{+}$ such that $0\leq\psi(x)\leq s(1,x)$ for all $x$ in $M_+$, there is $a\in M$ with $0\leq a\leq1$ such that $\psi(x)=s(a,x)$.
\end{dfn}

We will use the following basic property of the standard form.
\begin{thm}\label{C}
 Let $\varphi$ be a faithful normal state on a ($\sigma$-finite) von Neumann algebra $M$ in standard form.  
 Then the sesquilinear form $s_\phi$ defined for $y,x\in M$ by
 \begin{equation}\label{a57}
 s_\varphi(y,x)= \langle \xi_\varphi, JyJx  \xi_\varphi\rangle  \end{equation}
  is a   strictly positive definite self-polar form  on $M\times M$ such that
 $$s_\varphi (1,x)=\varphi(x)\qquad\forall x\in M.$$
 Moreover, 
a form   $\psi\in M_*$ satisfies $0\le \psi \le \lambda \phi $ ($\lambda>0$)
 if and only if there is a (uniquely defined) $y\in M$ with $0\le y\le 1$  such that
 $\psi(x)= \lambda\langle \xi_\phi , JyJ x  \xi_\phi \rangle$ for all $x\in M$.
\end{thm}
 
 See \cite[Th. 23.30 p. 398]{P6}
for  detailed indications of the references.

We make crucial use of the following fact, 
for which the main idea goes back to Effros and Lance 
in \cite[proof of Prop. 4.5]{EL}.

\begin{cor}\label{rc2} Assuming $M$ in standard form, let $\xi \in P^\natural$ and let 
$\phi\in M_*$ be defined by $\phi(x)=\langle \xi  ,   x  \xi  \rangle$ for all $x\in M$.
Assume that $\phi\in M_*$ is faithful (but not necessarily normalized).
 Suppose that we have a positive linear map $\psi : \ovl M \to M_*$
(we could also view $\psi$ as an anti-linear map on $M$) such that
$  \psi(\ovl 1) =   \phi $.
Then   there is a unique  positive unital linear map $V: M \to M$
such that 
$$\forall x,y \in M \quad \psi(\ovl y)(x)=  \langle \xi , JV(y)J x  \xi  \rangle   .$$
\end{cor}
 
\begin{proof}
We have  $0\le \psi(\ovl y) \le   \phi $ for any $\ovl  0\le \ovl y\le \ovl  1$, and hence by    Theorem \ref{C} (and an obvious scaling)
   there is a unique $V(y)\in M$ with $0\le V(y)\le 1$
such that $\psi(\ovl y)(x)=  \langle \xi , JV(y)J x  \xi  \rangle$ for all $x\in M$.
Note that $  \psi(\ovl 1) =   \phi $ guarantees that $V(1)=1$.
By the uniqueness of $V(y)$ it is easy to extend
$y\mapsto V(y) $ to obtain a unique positive unital linear map $V: M \to M$
such that 
$$\forall x,y \in M \quad \psi(\ovl y)(x)=  \langle \xi , JV(y)J x  \xi  \rangle   .$$
\end{proof}

We will also make  use of the following fundamental property
of self-polar forms due to Woronowicz and Connes.

\begin{thm}\label{wc} Let $s$ be a self-polar form  on a ($\sigma$-finite) von Neumann algebra such that $x\mapsto s(1,x)$ is faithful. 
Any normalized bipositive   form 
 $s'$  on $M \times M$ such that
  $s \le s'$ on the diagonal,
  must coincide with $s$
  on $M \times M$. 
 \end{thm}
 See \cite[Cor. 23.20]{P6} for details.

We will use the following consequence of Theorem \ref{wc}.
\begin{lem}\label{wc26} For any finite set  $(x_j)$ in a von Neumann algebra $M$,
we have
\begin{equation}\label{eoo2} 
\|\sum \ovl {x_j} \otimes x_j +    \ovl {x^*_j} \otimes x^*_j \|_{\max}= \sup \sum s(x_j,x_j)+ s(x^*_j,x^*_j), \end{equation}
where the sup runs over all separately normal, normalized, bipositive forms $s$
on $M \times M$, or equivalently over all separately normal
bipositive forms $s$ with $s(1,1)\le 1$.
\end{lem}
\noindent We refer to \cite[(23.15) p. 403]{P6} for the proof.

\section{Pietsch-type
factorization}

The next statement is a variation on the Pietsch-type
factorization for oh-summing maps described
in \cite{Poh}. We say that an embedding  $j: B\to B(\cl H)$
has infinite multiplicity if $\cl H= H \oplus H\oplus \cdots $
and $j(b)=b \oplus b\oplus \cdots $ ($b\in B$).

\begin{pro}\label{oo2} 
Let $B\subset B(\cl H)$ be an    operator space embedded
with infinite multiplicity and let $(M,\tau)$ be a tracial probability space.
Let us denote by $C_+$ the set of finite rank operators
$h\in B(\cl H)$ such that $h\ge 0$ and $\|h\|_2=1$.
 Let $v: B \to M$ be a  linear map such that
 $v(b^*)=v(b)^*$ for all $b\in B$ (i.e. $v$ is self-adjoint) and
for any finite set $(b_j)$ in $B$ such that
$t=\sum \ovl {b_j} \otimes b_j$ is self-adjoint
we have
\begin{equation}\label{eoo4}
\|\sum \ovl {v(b_j)}\otimes v(b_j) \|_{ \ovl{M }\otimes_{\max} M}
\le \|\sum \ovl {b_j} \otimes b_j \|_{ \ovl{B }\otimes_{\min} B  }.\end{equation}
Then there  is a net  $(h_\a)$  in $C_+$
 such that for any $b\in B$ the following limit exists and we have
\begin{equation}\label{eoo5}\tau(v(b)^*v(b)) \le \lim   \tr( h_\a b^* h_\a b).\end{equation}
\end{pro} 
\begin{proof} 
Since $v$ is self-adjoint and $\tau$ tracial, we have $\tau(v(b)^*v(b))=\tau(v(b^*)^*v(b^*))$. Then for any self-adjoint $t=\sum \ovl {b_j} \otimes b_j$   
$$\sum \tau(v(b_j)^*v(b_j)) \le (1/2)\| \sum \ovl {v(b_j)}\otimes v(b_j)
+ \ovl {v(b^*_j)}\otimes v(b^*_j)\|_{\max}
\le (1/2)\| \sum \ovl { b_j}\otimes  b_j
+ \ovl { b^*_j}\otimes b^*_j\|_{\min}$$
$$=  \sup_{h\in C_+}  (1/2)\sum  \tr( h  b_j^* h  b_j) +\tr( h  b_j h  b_j^*)
= \sup_{h\in C_+}\sum  \tr( h  b_j^* h  b_j) $$
where for the next to last equality we refer to   \cite[Lemma 22.22]{P6}
using the fact that $\|t\|_{\min}= \|t\|_{\max} $ for any $t=\sum \ovl {b_j} \otimes b_j
\in \ovl{B(\cl H)} \otimes B(\cl H)$. We expand on the latter fact in Remark \ref{roo1}.\\
By a well known variant of Hahn-Banach (see e.g. \cite[Lemma A16]{P6})
there is a net $(\mu_\a)$
of finitely supported probability measures on
$C_+$ such that for any $b\in B$ 
the following limit exist and we have
$$\tau(v(b)^*v(b)) \le \lim \int    \tr( h  b^* h  b) d\mu_\a(h) .$$
Then using the fact that $B$ is represented on $\cl H$ with infinite multiplicity
one can obtain  a net for which \eqref{eoo5} holds (see e.g. \cite[Prop. 4.23]{P6} for details).
\end{proof}
\begin{rem}\label{bla3} Let $B,M$ be $C^*$-algebras
and let $v: B \to M$ be a c.b. (resp. c.p.) map.
Then, for any $C^*$-algebra $A$, the map $Id_A \otimes v$ extends
to a bounded map from $A \otimes_{\min} B$ 
to $A \otimes_{\min} M$ (resp. from $A \otimes_{\max} B$ 
to $A \otimes_{\max} M$ with norm at most $\|v\|_{cb}$ (resp. $\|v\|$).
See e.g. \cite[\S 7]{P6}.
It follows that the map $\ovl v \otimes v$ extends
to a bounded map from $\ovl B \otimes_{\min} B$ 
to $\ovl M \otimes_{\min} M$ (resp. from $\ovl B \otimes_{\max} B$ 
to $\ovl M \otimes_{\max} M$ with norm at most $\|v\|^2_{cb}$ (resp. $\|v\|^2$).
\end{rem}

The next fact, which plays an important role in our paper, seems to have passed unnoticed.
\begin{lem}\label{28} Let $v: B \to D$ be a   positive map between   $C^*$-algebras. Then for any self-adjoint 
tensor $t=\sum \ovl {x_j} \otimes x_j\in \ovl B \otimes B$, we have
\begin{equation}\label{eoo3}\|\sum \ovl {v(x_j)}\otimes v(x_j) \|_{ \ovl{D }\otimes_{\max} D}
\le \|v\|^2\|\sum \ovl {x_j} \otimes x_j \|_{ \ovl{B }\otimes_{\max} B  }.\end{equation}
\end{lem}
\begin{proof} 
The map $v^{**}: B^{**} \to D^{**}$ is a normal   positive map.
We have $1\le v^{**}(1) \le \|v^{**}\| 1=\|v \| 1$.
Therefore for any separately normal,   bipositive
form $s: D^{**} \times D^{**}$ the composition
$s( v^{**}, v^{**})$ is a form with  the same properties.
By bipositivity we have $s( v^{**}(1), v^{**}(1)) \le s(1,1) \|v\|^2$.
By \eqref{eoo2} (and by homomogeneity) for any self-adjoint 
tensor $t=\sum \ovl {x_j} \otimes x_j\in \ovl {B^{**}} \otimes B^{**}$   we have
$$\|\sum \ovl {v^{**}(x_j)} \otimes {v^{**}(x_j)}  \|_{ \ovl{D^{**}}\otimes_{\max} D^{**} }
\le \|v\|^2
\|\sum \ovl {x_j} \otimes x_j   \|_{ \ovl{B^{**}}\otimes_{\max} B^{**} }.$$
But since we have a canonical isometric embedding
${ \ovl{D }\otimes_{\max} D }\subset { \ovl{D^{**}}\otimes_{\max} D^{**} } $
(and similarly for $B$) 
we obtain \eqref{eoo3}.
\end{proof}

\begin{rem}\label{roo1}
When $B$ has the WEP, in particular when $B=B(H)$,
for any $(b_j)$ in $B$ 
we have
\begin{equation}\label{eb}
\| \sum \ovl { b_j}\otimes  b_j
 \|_{\min}=\| \sum \ovl { b_j}\otimes  b_j
 \|_{\max}.\end{equation}
 See \cite[Cor. 22.16]{P6} for a detailed proof.
Thus when $B=B(H)$ (or when $B$ has the WEP)
 \eqref{eoo3} becomes for any self-adjoint $\sum \ovl {x_j} \otimes x_j\in \ovl B \otimes B$
  \begin{equation}\label{eoo3'}
  \|\sum \ovl {v(x_j)}\otimes v(x_j) \|_{ \ovl{D }\otimes_{\max} D}
\le \|v\|^2\|\sum \ovl {x_j} \otimes x_j \|_{ \ovl{B }\otimes_{\min} B  }.\end{equation}
Moreover, if $v$ is c.p. with $\|v\|\le 1$
then \eqref{eoo3'} holds for any   $(x_j)$ in $B$
by Remark \ref{bla3}.
\end{rem}

\begin{rem}\label{ch} Let $R$ and $C$ denote as usual the row and column
operator spaces. It is known (see \cite[Rem. 22.20]{P6}) that if $v: B \to D$
is such that $\|Id_X \otimes v : X \otimes_{\min} B \to X \otimes_{\min} D\|\le 1$ both for $X=R$ and $X=C$ then $v$ satisfies \eqref{eoo3}.
 Therefore, by Choi's generalization of Kadison's inequality (\cite[Cor. 2.8]{Choi}), any $2$-positive contraction satisfies \eqref{eoo3}.
 Given this, the preceding lemma is not so surprising.
In particular, any complete contraction $v$ satisfies \eqref{eoo3}.
See \cite[Cor. 22.19]{P6} for a detailed proof.
\end{rem}
\begin{rem}\label{rc} 
Let $u,v$  be such that $Id_M=vu$ (as in \eqref{e1}).
Assume that there are constants $c_u$ and $c_v$
such that for any $n$ and  $(x_j)_{1\le j\le n}$ in $M$
and any $(b_j)_{1\le j\le n}$ in $B(H)$ with 
$\sum \ovl {x_j} \otimes x_j$ and $\sum \ovl {b_j} \otimes b_j$ self-adjoint,
we have
\begin{equation}\label{eoo4}\|\sum \ovl {u(x_j)}\otimes u (x_j) \|_{ \ovl{B(H) }\otimes_{\min} B(H)}
\le c_u\|\sum \ovl {x_j} \otimes x_j \|_{ \ovl{M }\otimes_{\min} M }.\end{equation}

\begin{equation}\label{eoo5}\|\sum \ovl {v(b_j)}\otimes v(b_j) \|_{ \ovl{M }\otimes_{\max} M}
\le c_v\|\sum \ovl {b_j} \otimes b_j \|_{ \ovl{B(H) }\otimes_{\max} B(H)  }.\end{equation}
Then $M$ is injective.
Indeed,
by \eqref{eb}
we have
$$\|\sum \ovl {x_j} \otimes x_j \|_{ \ovl{M }\otimes_{\max} M }
\le c_uc_v\|\sum \ovl {x_j} \otimes x_j \|_{ \ovl{M }\otimes_{\min} M } ,$$
which  by a theorem due to Haagerup (see \cite[Th. 23.7]{P6})
implies that $M$ has the WEP and hence (being a von Neumann algebra) that    $M$ is   injective.\\
In particular 
this shows that if
$Id_M=vu$ with
 $u$  c.b. and $v$ c.p.  
then
\eqref{eoo4}  and \eqref{eoo5} hold (see Remark \ref{bla3}) and hence
 $M$ is injective. Actually, much less suffices, as we now explain.\\
Let $X$ be an operator space, for example 
$R$ or $C$.
Let us say that $v: B \to D$ is $X$-bounded if 
$\|Id_X \otimes v : X \otimes_{\min} B \to X \otimes_{\min} D\|<\infty$.
It is known (see \cite[p. 375, Rem. 22.20]{P6})
that if $v$ is $X$-bounded both for  $X=R$ and $X=C$,  or equivalently
for $X=R\oplus C$,
then $v$ satisfies \eqref{eoo5} (the contractive case is discussed in Remark \ref{ch}).\\
As for \eqref{eoo4}, assuming
$M\subset B(\cl H)$ completely isometrically, it suffices for the same reason that $u$ admits an extension 
$\tilde u: B(\cl H)\to B(H) $ that is  $X$-bounded both for  $X=R$ and $X=C$.
Actually, it suffices that $\tilde u$ be $X$-bounded when $X$ is the operator Hilbert space $OH$
(see \cite[p. 122]{P4}).

This leads to what seems to be an interesting example.
Let  $M$ be a \emph{non injective} von Neumann algebra.
Let   $u,v$ be maps 
 such that  $Id_M= vu$ (as in \eqref{e1}). Assume that $u: M \to B(H)$ is   2-positive  
and    that
$v$ satisfies \eqref{eoo5} (as is the case for e.g. $M=L(\F_\infty)$).
Then    $u: M \to B(H)$ is an example of a map
that is $X$-bounded for $X=R\oplus C$ but \emph{does not extend}
to a similarly bounded map $\tilde u: B(\cl H)\to B(H) $.
\end{rem}

 \section{Proof of main result}
 
 We start by some simple preliminary elementary propositions.
 
 \begin{pro}\label{fac} Let $T: M \to N$ be a linear map between von Neumann algebras.
 Assume that there is a net of pairs of positive contractions $u_\a: M \to B(H_\a)$, 
 $v_\a: B(H_\a) \to N$, with $u_\a$ normal, such that $v_\a u_\a \to T$ pointwise-weak*. 
 Then there is (for some $H$) a pair of positive contractions $u : M \to B(H )$, 
 $v : B(H ) \to N$,  with $u$ normal,  such that $vu=T$.
 The construction is such that  if the $u_\a$'s and $v_\a$'s are all unital then so are $u,v$, and if 
 the $v_\a$'s are all c.p. then so is $v$.
 \end{pro}
  \begin{proof}
  Let $B=\oplus_\a B(H_\a)$. Let $u: M \to B$ be defined by $u(x)=\oplus_\a u_\a(x)$ ($x\in M$). Let $p_\a: B \to B(H_\a)$ be the $\a$-th coordinate mapping.
  We define $v: B \to N$ by setting
  $v(b)=\text{weak*}\lim_\cl U v_\a p_\a(b)= \text{weak*}\lim_\cl U v_\a (b_\a)$ for any $b=(b_\a)\in B$,
  where $\cl U$ is an ultrafilter refining the net. Clearly, $vu=T$
  and all the announced properties hold.
Lastly, as   the identity of $B$ factors
 through $B(H)$ with $H=\oplus_\a H_\a$ via unital c.p. normal maps,
 we may replace $B$ by $B(H)$.
  \end{proof}
The next lemma is elementary.
 \begin{lem}\label{L2} 
 Let $u: M \to B(H)$ and $v:B(H) \to M$
 be positive maps (resp. with $u$ normal). Let $\t=vu:M \to M$. 
 Let $ \vp
 >0$ and $0\le \d <1$.
 If $\|\t(1)-1\|\le \d$, there are unital positive maps $u' : M \to B(H)$ and $v' :B(H) \to M$
 (resp. with $u' $ normal) such that 
 $$\|\t-v'  u' \| \le    (2^{1/2}+1)(\d+\vp\|v\|)^{1/2}  +\vp\|v\|.$$
 Moreover, if in addition $v$ is c.p. we obtain $v'$ also c.p.
\end{lem}
 \begin{proof} 
Let $f\in M_*$ be a   normal state. Then   let
$u_\vp(x)= u(x) + \vp f(x) 1_{B(H)} $ and let $u '(x)= u_\vp(1)^{-1/2} u_\vp(x)u_\vp(1)^{-1/2}$,
and $v_\vp(b) = v(u_\vp(1)^{1/2} b u_\vp(1)^{1/2} ) $ for $x\in M$ and $b\in B(H)$.
Note \begin{equation}\label{f13}
v_\vp u' = vu_\vp=\t + \vp f(\cdot) v(1) .\end{equation}
Then $u' : M \to B(H) $ is unital positive and normal, $v_\vp$ is positive
and 
$v_\vp(1) = \t(1_M)+\vp v(1_{B(H)})  $,
the latter being invertible
since $v_\vp(1) \ge \t(1_M)$ and $\d <1$.
Let $\eta=v_\vp(1) $.
Let $v'(b)=  \eta^{-1/2} v_\vp(b) \eta^{-1/2}$,
which is unital and positive. Equivalently
we have $v_\vp(b)=\eta^{1/2} v'(b) \eta^{1/2}$, and hence by \eqref{f13}
$$\t=   \eta^{1/2} v' u' \eta^{1/2} - \vp f(\cdot) v(1).$$
This implies (we write $ \eta^{1/2} v' u' \eta^{1/2} -v' u' = (\eta^{1/2} -1)v' u' \eta^{1/2}+   v' u' (\eta^{1/2} -1)$)
$$\|\t- v' u'\|\le   \|\eta^{1/2}-1\| \|\eta^{1/2}\|+\|\eta^{1/2}-1\| + \vp \|v(1)\|.  $$
Since $1-\d\le \eta(1) \le 1+\d +\vp \|v(1)\|$, we have
$\|\eta^{1/2}-1\|\le (\d +\vp \|v\|)^{1/2}$ and $\|\eta^{1/2}\| \le (1+\d +\vp \|v \|)^{1/2} \le 2^{1/2}$.
The announced bound follows.
\end{proof}

\begin{rem}\label{fac1} Let $M$ be a von Neumann algebra.
 Assume that we have a net of pairs of maps $u_\vp: M \to B(H)$, $v_\vp: B(H) \to M$
 (here for convenience $\vp$ is used for the index of the net)
 such that the composition $\theta_\vp=v_\vp u_\vp $ 
 satisfies
    $\|\theta_\vp(1) - 1\|_M  \to 0$ and 
    $ \theta_\vp(x) \to  x $ weak*   for any $x\in M$. \\
    Note that this holds in particular if $v_\vp u_\vp =Id_M$ independently of $\vp$.\\
Assume moreover that    $u_\vp$, $u_\vp$ are both positive, that  $\sup_\vp\|v_\vp\|<\infty$   
 and that $u_\vp$ is normal.
Then there is a factorization  $Id_M=vu$ as in \eqref{e1}
with $u$ normal and $u,v$ both \emph{unital} positive.
Indeed, this follows from Lemma \ref{L2} and Proposition \ref{fac}.\\
In particular, having $Id_M=v'u'$ as in \eqref{e1} with $u', v'$ merely positive and $u'$ normal
implies the existence of  $Id_M=vu$ as in \eqref{e1}
with $u$ normal and $u,v$ both \emph{unital} positive.\\
Furthermore we will show at the end of the proof of Theorem \ref{ot1} that the latter factorization
implies that $Id_M$  admits the seemingly injective factorization of Definition \ref{d1}.
\end{rem}

 As before, for any map $u:A \to M$ from a $C^*$-algebra to a von Neumann one
we denote by $\ddot u: A^{**} \to M$ the canonical normal
extension of $u$ to $ A^{**}$, so that $\ddot u i_A=u$.

We next state an easy fact  that will help us to relate our main topic
to lifting problems.
 \begin{lem}\label{23/10} If a 
von Neumann  algebra $M$ is QWEP, there is a  unital 
$C^*$-algebra $W$ and a surjective   unital $*$-homomorphism
$q: W \to M$ that admits  for some $H$  a factorization of the form
$$ q   : W  {\buildrel w_1\over\longrightarrow} B(H) {\buildrel w_2\over\longrightarrow} M  $$
where $w_1$ is a   unital $*$-homomorphism and $w_2$ is
a   unital  c.p. map (and hence $\|w_2\|_{cb} = 1$).
\end{lem}
\begin{proof}   Let $q: W \to M$ be a surjective unital $*$-homomorphism.
Let $i_W: W \to W^{**}$ denote the inclusion.
 Let $\ddot q: W^{**}  \to M$ denote the 
 normal $*$-homomorphism extending $q$ so that $\ddot q i_W =q$. 
If $W$ is WEP, the inclusion
 $i_W: W \to W^{**}$ factors as 
 $ i_W   : W  {\buildrel w_1\over\longrightarrow} B(H) {\buildrel v_2\over\longrightarrow} W^{**}  $,
 with $w_1$ unital $*$-homomorphism and $v_2$ unital c.p.  
 Let   $w_2=\ddot q v_2:B(H) \to M$. We have then $q=w_2w_1$.
\end{proof}
 
 The next lemma is a simple 
 consequence of Corollary \ref{Mi}. 
 \begin{lem}\label{22/10}
 Let $M$ be a  von Neumann algebra that is QWEP.
 Let $T: M \to M$ be a finite rank normal map.
 For some $H$ there is a factorization of $T$ of the form
 $$ T   : M  {\buildrel u\over\longrightarrow} B(H) {\buildrel v\over\longrightarrow} M  $$
 where $u$ is  normal, finite rank  with $\|u\|=\|T\|$ and $v$ is 
 a unital c.p. map (and hence $\|v\|=1$).\\
 If $T$ is in addition unital positive, then we can ensure that $u$ is also unital positive.
 \end{lem}
 \begin{proof} Let $W,q,w_1,w_2$ be as
 in the preceding lemma.
 Let $E=T(M) \subset M =q(W)$.   By Corollary \ref{Mi} there is a contractive unital positive lifting
$r: E  \to W$. Let $u=w_1 r T$ and $v=w_2$.
Then since $T$ is normal (resp. and unital positive) and $\dim(E)<\infty$
 the map $u$ is 
  normal (resp. and unital positive) with $\|u\|\le \|T\|$,
  while  $v$ is u.c.p.
and $T=vu$. The latter implies $\|T\|\le \|v\|\|u\|\le \|u\|$.
  \end{proof}

$$\xymatrix{ & & W\ar[r]^{   {w_1} }\ar[dr]^{q}&B(H)\ar[d]^{\bf v} &   \\
& M\ar@/^3pc/[urr]^{\bf u}\ar[r]^{T}&T(M)\ar[u]^{r } \ar@{^{(}->}[r] & M  
   &   }$$
 
 Before proving Theorem \ref{ot1},
 we start by some preliminary remarks.
 
 \begin{rem}\label{rf2}
Consider the implication (i) $\Rightarrow$ (ii) in Theorem \ref{ot1}.
 By structural results   we may 
assume that $M$ is $\sigma$-finite.
Indeed, by  \cite[Chap. 5, prop. 1.40]{Tak} a general $M$ is isomorphic
to a direct sum of a family of algebras of the form
$M=B(H) \bar\otimes N$ with $N$ $\sigma$-finite.
For such an $M$, 
there is a net of projections $(p_\a)$ in $M$ tending weak* to $1$
such that $p_\a M p_\a$ is $\sigma$-finite,
or equivalently the unital subalgebra $M_\a=p_\a M p_\a \oplus \CC (1-p_\a)$ is 
$\sigma$-finite and is the range of a unital normal c.p. projection $Q_\a$
defined by $Q_\a(x)= p_\a x p_\a + f(x) (1-p_\a)$ where $f$ is a normal state on $M$.
Then $Q_\a$ tends  pointwise weak* to $Id_M$ 
and the range of each $Q_\a$ is $\sigma$-finite.
Moreover, if $M$ is seemingly injective, so is
 $Q_\a(M)$ by Remark \ref{rf1}.  
Thus to show that $M$ has the weak* PAP it clearly suffices to
show that each $Q_\a(M)$ has it, and
 the latter are $\sigma$-finite.
 \end{rem} 
\begin{rem}
Consider again the implication (i) $\Rightarrow$ (ii) in Theorem \ref{ot1}.
Assume $M$ $\sigma$-finite.
A tempting approach to prove this without using standard forms 
is to 
  first prove this assuming $M$ finite,
  from which the   semi-finite case  
    follows easily
 (using $pMp$
for finite   projections $p$), and then
   to  invoke Takesaki's duality theorem
to pass from the semi-finite
case to the general one.
This roughly produces a general $\sigma$-finite  algebra
 as the range of a unital c.p. projection on a 
 semi-finite one. However since the latter projection is 
 in general not normal,   this route 
 does not seem to go through
  to prove the weak* PAP.
  (We fell into that trap in a previous version of this paper.)
   In \cite{Wac}, to show  
   that injectivity implies the weak* CPAP, Wassermann uses this
  path to prove that injectivity implies (iii) in Theorem \ref{bla}
  and then he relies on
    (iii) $\Rightarrow$ (i) in Theorem \ref{bla}
  (previously proved by Effros-Lance \cite{EL}).
  In the seemingly injective context,
  the analogue of the latter implication
  is what we are after, so this approach no longer seems
  to work. 
 This explains our recourse to standard forms.
 Using the latter, we can give an argument covering
 the case of a general $\sigma$-finite  algebra.
 Nevertheless, we include the finite case in the proof below
 because we feel  it will be much easier for 
 the reader to follow the general case after this warm up.
 \end{rem} 

     \begin{proof}[Proof of Theorem \ref{ot1}]
Assume (i).  Then $M$ is  isometric to a (Banach) quotient of some $B(H)$
and hence is QWEP by Kirchberg's results (see e.g.  \cite{Oz1} or \cite[Th. 15.5]{P6}).
We will derive the   matricial weak*  PAP  
by the same kind of argument as the one   which   shows that
injective von Neumann algebras have the
weak* CPAP
(see Remark \ref{co} for references).\\
By Remark \ref{rf2}  we may 
assume that $M$ is $\sigma$-finite.
 
 We will first prove  (i) $\Rightarrow$ (ii) in the finite  
 case. 
Then in the second part of the proof we will show
how the argument can be adapted 
to cover the general case using 
the standard form of $M$.

Thus  we first consider
   a tracial probability space $(M,\tau)$.
  Let $B(H), u,v$ be as in   \eqref{e1}  with $u,v$ both unital and positive
  (or equivalently unital and contractive).
  Let ${\cl H}=H\oplus H \oplus \cdots$.
  We may embed $B(H)$ in $B({\cl H})$
  acting diagonally with infinite multiplicity,
  and we have  a unital c.p. projection from $B(\cl H)$ to  this copy of  $B(H)$.
  By Proposition \ref{oo2} (with $B=B(H)$) and Remark \ref{roo1}
we may assume 
that we have  an ultrafilter ${\cl U}$ refining a net of  finite rank operators
 $h_\a\in C_+$ such that
for all $b\in B(H)$
\begin{equation}\label{282}
\tau(v(b)^*v(b))\le \lim\nl_{\cl U} \tr( b^* h_\a b h_\a).\end{equation}
Here since we assume (i) $v$ is unital and c.p. and hence  satisfies \eqref{eoo4}
by Remark \ref{bla3} and \eqref{eb}. However, 
 we wish to record for future use (at the end of the proof)  that this also holds 
for all positive $v: B(H) \to M$  by Lemma \ref{28} and \eqref{eb}.

Let $b'=u(U)$ with $U$ unitary in $M$. We have $v(b')=U$ and hence
$1=\tau(v(b')^*v(b'))\le \lim\nl_{\cl U} \tr( {b'}^* h_\a b' h_\a)\le 1$
so that  $\tau(v(b')^*v(b'))= \lim \tr( {b'}^* h_\a b' h_\a)$.\\
(Therefore
$\lim \| b' h_\a -h_\a b'\|_2=0$, but actually we will not use this.)
\\Let $\langle\langle b,b\rangle\rangle=  \lim\nl_{\cl U} \tr( b^* h_\a b h_\a)-\tau(v(b)^*v(b))$
   for $b\in B(H)$. 
   By Cauchy-Schwarz
   for the inner product $\langle\langle b,b\rangle\rangle$ we have
   $$|\langle\langle b',b\rangle\rangle|^2 \le \langle\langle b',b'\rangle\rangle\langle\langle b,b\rangle\rangle$$
   and hence $\langle\langle b,b'\rangle\rangle=0$ for any $b\in B(H)$.
   Since $M$ is  spanned by its unitaries this remains true for any
     $b'\in u(M)$.
   In particular
\begin{equation}\label{eoo1}
\tau(v(b)^*v(b))= \lim\nl_{\cl U} \tr( b^* h_\a b h_\a)\end{equation}
   for any $b\in u(M)$.
    Note that
by polarization \eqref{eoo1}  implies
$\tau(v(b)^*v(a))= \lim \tr( b^* h_\a a h_\a)$ for any pair $a,b\in u(M)$.
 Moreover,  for any $b\in B$ and    $x \in M$,  let $b'=u(x)$.   Since $\langle\langle b,b'\rangle\rangle=0$ we have
\begin{equation}\label{eoo11}\tau(x^*v(b))= \lim\nl_{\cl U} \tr( u(x)^* h_\a  b h_\a).\end{equation}
Recall that  $u: M \to B(H)$ is a normal unital positive isometric map and
 all the $h_\a$'s are of finite rank.
Since $u$ is normal and $h_\a b h_\a\in B(H)_*$, 
for any given $b\in M$, 
the linear form $x\mapsto \ovl{\tr( u(x)^* h_\a b h_\a)}$
is in $M_*$. We identify $M_*$ with $L_1(\tau)$. Note that since $u$ is positive $u(x)^*=u(x^*)$. Thus
we have   a (uniquely defined) linear map $V_\a: B(H) \to M_*$ such that
\begin{equation}\label{e10}
\forall x\in M\ \forall b\in B(H) \quad \tr( u(x)^* h_\a b h_\a) =\tau(x^* V_\a(b)).\end{equation}
 Since $h_\a$ has finite rank, so has $V_\a$.
Since  $u$ is positive, $\tau(x^* V_\a(b))\ge 0$ for any $x\in M_+$, $b\in B(H)_+$,
and hence $V_\a$ is positive. By \eqref{eoo11}
$$\tau(x^* v(b))=\lim \tau(x^* V_\a(b))$$
for any $x\in M$, $b\in B(H)$
and hence denoting by $j: M \to M_*$ the canonical embedding
we have $V_\a(b)-j(v(b)) \to 0$ with respect to $\sigma(M_*,M)$
or equivalently weakly in $M_*$.
Passing to convex combinations of the $V_\a$'s   we may assume
that $\|V_\a(1_{B(H)}) - j(1_M)\|_{M_*}\to 0$ (and also that $\|V_\a(b) - j(v(b))\|_{M_*}\to 0$ for any $b\in B(H)$). Assume for a moment that
$V_\a(1)^{-1}$ exists and lies in $M$.
Then if we set 
\begin{equation}\label{e11}
T_\a(b)= V_\a(1)^{-1/2} V_\a(b) V_\a(1)^{-1/2},\end{equation}
the map  $T_\a : B(H) \to M$ is positive, unital (and hence contractive), of finite rank 
and we have
$$ \tau(x^*v(b))=\lim \tau(x^* V_\a(1)^{1/2}T_\a(b) V_\a(1)^{1/2}) .$$
Since $V_\a(1) \to 1$ in $L_1(\tau)$
we have
$\lim\| V_\a(1)^{1/2}-1\|_{L_2(\tau)} \to 0$; indeed this follows from
  the Powers-St\o rmer inequality (see e.g. \cite[Prop. 1.2.1]{Co},
 but in the present case since $V_\a(1)^{1/2}$ and $1$ commute
this could be reduced to the commutative case). In any case (for $x\in M$, $b\in B(H)$)
$$ \tau(x^*v(b))=\lim \tau(x^*  T_\a(b)  ) .$$
This shows that $T_\a(b) \to v(b)$ for $\sigma(M,M)$
and hence since $M$ is dense in $M_*$ and $(T_\a)$ is equicontinuous, $T_\a(b) \to v(b)$ for $\sigma(M,M_*)$. 
\\
Now let $H_\a\subset H$ be the (finite dimensional) range of $h_\a$
and let $p_\a\in B(H,H_\a) $ denote the orthogonal projection onto $H_\a$.
We define  $u_\a  :M \to  B(H_\a) $ by setting 
$u_\a (x)= p_\a u(x) p_\a^*$, 
and $v_\a (y)= T_{\a} (p_\a^* y p_\a)$. 
Note 
that since 
$p_\a^* p_\a h_\a = h_\a p_\a^* p_\a=h_\a$, by
\eqref{e10} and  \eqref{e11} we have
$T_{\a} (p_\a^* p_\a b p_\a^* p_\a) =T_\a(b)$ for any $b\in B(H)$,
and hence 
$v_\a(u_\a(x))=T_\a(u(x))$ for any $x\in M$.
This shows that $v_\a(u_\a(x))$ tends weak* to $  v(u(x))=x$.
Let $n(\a)={\dim(H_\a)}$.  We may identify $B(H_\a)$ with $ M_{n(\a)}$.
Then  $u_\a$ is normal,  $u_\a, v_\a$ are both unital and positive 
(and hence   contractive).
Thus $M$ satisfies (ii).
The only drawback is that we assumed $V_\a(1)^{-1}\in M$.
 This can be fixed easily: we pick a state $f$ on $B(H)$,    we replace $V_\a$
by $V_{\a,\vp}$ defined by
$V_{\a,\vp}(b)= V_\a(b) +\vp f( b) 1_M$ ($b\in B(H))$
 so that $V_{\a,\vp}(1)^{-1} \le \vp^{-1}   1_M$, 
and we let $\vp>0$ tend to zero as part of our net.
This completes the proof that (i) $\Rightarrow$ (ii) in the finite case. 

We will now prove the same implication (i) $\Rightarrow$ (ii) for a general 
  $M$. Assume (i).
Again, by Remark \ref{rf2} we may assume $M$ $\sigma$-finite so that $M$ admits a
faithful normal state $\varphi$. 

We will use Corollary \ref{rc2}.  Let $\xi=\xi_\phi$. Note $\|\xi\|=1$.
Let $s(x,x)=\langle \xi, Jx Jx \xi\rangle$ ($x\in M$).\\
Let $B=B(H)$. We claim that for any self-adjoint   tensor 
$t= \sum \ovl{b_j} \otimes {b_j}$   
in $\ovl B \otimes B$, we have
$$\sum \langle \xi, J v(b_j)J v(b_j) \xi\rangle 
\le  \| \sum \ovl{ b_j} \otimes { b_j} \|_{\min}  .$$
 Indeed, by Remark \ref{cy2} and   \eqref{eb} we have
 $$\sum \langle \xi, J v(b_j)J v(b_j) \xi\rangle 
\le \|\sum \ovl{v(b_j)} \otimes {v(b_j)} \|_{\max}\le \|\sum \ovl{b_j} \otimes {b_j} \|_{\max}
=\|\sum \ovl{b_j} \otimes {b_j} \|_{\min} .$$ 
  By the argument used in Proposition \ref{oo2}, our claim implies that
there is a net $( h_\a)$   in $  C_+$
such that  
  \begin{equation}\label{d30}
\forall b\in B\quad   \langle \xi, Jv(b) J v(b) \xi\rangle 
\le \lim\nl_{\cl U}    \tr (b^* h_\a {b} h_\a ) ,
\end{equation}
where  $\cl U$ is an ultrafilter  refining our net.
Indeed,  we  may use the diagonal  embedding $B(H)\subset B(H\oplus H\oplus \cdots)$ to replace $H$ by $H\oplus H\oplus \cdots$ in order
to replace the finitely supported probabilities on $\cl P\times C_+$
by  singly supported ones. (Actually this is not essential,
but it makes the argument slightly easier to follow.)
\\
Since $vu=Id_M$, we have
  \begin{equation}\label{d30'}
\forall x\in M\quad   \langle \xi, Jx J x \xi\rangle 
\le \lim\nl_{\cl U}    \tr ( u(x)^* h_\a {  u(x) } h_\a ) ,
\end{equation}
We will now show that equality holds in \eqref{d30'}. 
Let $s'$ be the sesquilinear form  defined on $M$ by
 $$s'(y,x)=\lim\nl_{\cl U}    \tr (u(y)^* h_\a {  u(x) } h_\a ) .$$
 Note that $s'$
is bipositive and $s'(1,1)=1$.
Since $s$ is self-polar 
and   $s \le s'$ on the diagonal,
 we must have $s=s'$  by Theorem \ref{wc}.
Thus we have  
\begin{equation}\label{d30b}\forall y,x\in M\quad   \langle \xi, Jy Jx \xi\rangle 
= \lim\nl_{\cl U}    \tr (u(y)^* h_\a {  u(x) } h_\a ) .\end{equation}
 Since 
  $h_\a$ has finite rank,
  there is a finite rank projection $P_\a$ such that
  $P_\a h_\a= h_\a P_\a=h_\a$, and hence 
  $\tr (  u(y ) ^* h_\a {  u(x) } h_\a )= \tr (  P_\a u(y ) ^* P_\a h_\a P_\a {  u(x) } P_\a h_\a )$.
  Therefore there is a positive unital  
  finite rank map $T_\a:  \ovl M \to M_*$ such that 
  \begin{equation}\label{d31} \forall y,x\in M\quad T_\a(\ovl y) (x)
  = \tr (  u(y ) ^* h_\a {  u(x) } h_\a ).\end{equation}
  We have $\lim_\cl U T_\a(\ovl y) (x)= \langle \xi, Jy Jx \xi\rangle$.
  In particular $T_\a(\ovl 1) \to \phi$ with respect to
  the weak topology $\sigma(M_*,M)$. Passing
  to convex linear combinations of the $T_\a$'s we may assume that
  $\|T_\a(\ovl 1) - \phi \|_{M_*}$ tends to $0$.
 \\ Let $\vp>0$. Consider the   positive map
  $T_{\a,\vp} :\ovl M \to M_*$ defined by
  $$T_{\a,\vp}(\ovl y)  =    T_\a(\ovl y)   +\vp \phi(y^*)\phi .$$
  The advantage of this perturbation of $T_{\a} $ is that
  $T_{\a,\vp}(\ovl 1)$ is faithful (since it dominates  $\vp  \phi$).
  Let $\xi_{\a,\vp} $ be the (unique)   vector in ${P}^\natural$ such that
  $T_{\a,\vp}(\ovl 1) (x)= \langle \xi_{\a,\vp},  x \xi_{\a,\vp} \rangle$.
  We now make $\vp$ tending to $0$ as part of our new net $(T_{\a,\vp})$
  indexed by ${\a,\vp}$. Clearly $\lim\| T_{\a,\vp}(\ovl 1)- \phi\|_{M_*}=0$.
  By the generalized Powers-St\o rmer inequality
  (see \cite[Lemma 2.10]{Hasta}) we have
  $\|\xi_{\a,\vp} -\xi \|\to 0$.
  \\
  We now apply Corollary \ref{rc2}
  to the faithful form $T_{\a,\vp}(\ovl 1)$.
  Since $T_{\a,\vp}$ is positive,
  there is a uniquely defined  positive unital finite rank
  map $V_{\a,\vp} : M \to M$ such that
  $T_{\a,\vp}(\ovl y)(x)=\langle \xi_{\a,\vp}, J V_{\a,\vp}(y) Jx \xi_{\a,\vp}\rangle $.  
  Since $\|\xi_{\a,\vp} -\xi \|\to 0$, we have
  $$\lim|\langle \xi_{\a,\vp}, J V_{\a,\vp}(y) Jx \xi_{\a,\vp}\rangle  -
  \langle \xi , J V_{\a,\vp}(y) Jx \xi \rangle| =0.$$
  Thus since $\lim T_{\a,\vp}(\ovl y)(x)=\lim T_{\a}(\ovl y)(x)$ we obtain
  $$\lim  \langle \xi , J V_{\a,\vp}(y) Jx \xi \rangle
  = \langle \xi, Jy Jx \xi\rangle ,$$
  and hence (since $\langle x_1\xi, Jy Jx_2 \xi\rangle= \langle \xi, Jy Jx_1^*x_2 \xi\rangle$)
  \begin{equation}\label{f8} \forall x_1,x_2\in M\forall y\in M\quad
  \lim  \langle x_1\xi , J V_{\a,\vp}(y) Jx_2 \xi \rangle
  = \langle x_1\xi, Jy Jx_2 \xi\rangle.\end{equation}
  Since $V_{\a,\vp}$ is unital and positive it is contractive.
  By \eqref{d31} the form $(y,x) \mapsto T_{\a,\vp}(\ovl y)(x)$ is separately normal 
  (since $u$ and $\phi$ are normal).
  Let 
   $ L=\{\phi_x\mid x\in M\} \subset M_*$ with
   $\phi_x(  \ovl y)=  \langle \xi , JyJ x \xi  \rangle$. 
  The linear span of the set $L$ is norm dense in $\ovl{M_*}$ (because 
  by Remark \ref{cy}
  $\ovl{M\xi}=H$ 
   which implies $L^\perp=\{0\}$). Therefore $V_{\a,\vp}$ is normal
  and \eqref{f8}  implies that
  $\lim V_{\a,\vp}(y)=y$ for the weak* topology $\sigma(M,M_*)$.
   This proves that (i) $\Rightarrow$ (iii) (since we already noted that
   QWEP follows from (i)).
   We recorded this argument  for later use in Theorem \ref{ot2},
   but we actually want to prove (i) $\Rightarrow$ (ii).
   This requires a little more work. Assume (i) still.

 We start again from
  \begin{equation}\label{d30''}
\forall x\in M\quad   \langle \xi, Jx J x \xi\rangle 
\le \lim\nl_{\cl U}    \tr (  u(x)^* h_\a   u(x)  h_\a ) ,
\end{equation}
Arguing exactly as for \eqref{eoo11} we deduce from \eqref{d30} and \eqref{d30b} that 
$$\forall  x\in M, \forall b\in B\quad   \langle \xi, Jv(b) Jx \xi\rangle 
= \lim\nl_{\cl U}    \tr (   b ^* h_\a    u(x)   h_\a ) .$$
Since $u$ is normal and $h_\a   b^* h_\a \in B(H)_*$ 
for any   $b\in B$, 
the linear form $$\theta_\a(\ovl b): x\mapsto  \tr (   b^* h_\a  u(x)  h_\a )$$
is in $ {M_*}$.  Moreover $\theta_\a: \ovl B \to  {M_*}$ is a (linear) positive map of finite rank.
Recall  $P_\a$  are finite rank projections in $B(H)$ such that
$h_\a= P_\a  h_\a P_\a $.
Note that
$$\theta_\a(\ovl b)=\theta_\a(\ovl{ P_\a b P_\a}).$$
Let $\vp>0$. Let $f_\a$ be a normal state on $B$ such that $f_\a(P_\a)=1$.
Let 
\begin{equation}\label{d33}\theta_{\a,\vp}(\ovl b)  = \theta_\a(\ovl b) + \vp f_\a(b^*) \phi .\end{equation}
We still have 
\begin{equation}\label{d33c} \theta_{\a,\vp}(\ovl b)=\theta_{\a,\vp}(\ovl{ P_\a b P_\a})\end{equation}
 and 
 $\theta_{\a,\vp}: \ovl B \to  {M_*}$ is   positive with finite rank,
but in addition  $ { \theta_{\a,\vp}( \ovl 1)}$ is faithful.

Let $\xi_{\a,\vp} $ be the (unique)   vector in ${P}^\natural$ such that
  $\theta_{\a,\vp}(\ovl 1) (x)= \langle \xi_{\a,\vp},  x \xi_{\a,\vp} \rangle$.
  We now apply Corollary \ref{rc2}
  to the faithful form $\theta_{\a,\vp}(\ovl 1)$.
  Since $\theta_{\a,\vp}$ is positive,
  there is a uniquely defined  positive unital linear  
  map $W_{\a,\vp} : B \to M$ such that
   \begin{equation}\label{d32} \theta_{\a,\vp}(\ovl b)(x)=\langle \xi_{\a,\vp}, J W_{\a,\vp}(b) Jx \xi_{\a,\vp}\rangle .\end{equation}
  We will make $\vp$ tending to $0$ as part of our new nets
  indexed by ${\a,\vp}$. Since $\theta_{\a,\vp}(\ovl 1) \to \phi$ weakly in $M_*$,
  we may pass to convex combinations to ensure that $\theta_{\a,\vp}(\ovl 1) \to \phi$
  in norm in $M_*$. Then as before $\|\xi_{\a,\vp} -\xi \|\to 0$.
   
 Let 
   $T_{\a,\vp}=W_{\a,\vp}u $.
   By \eqref{d33c} we have $W_{\a,\vp}(b)=    W_{\a,\vp}(P_\a b P_\a)$ for all $b\in B$.      
   In particular, $W_{\a,\vp}(1)=    W_{\a,\vp}(P_\a)=1$. 
 Let $H_\a=P_\a(H)$. 
 Let $i_\a: H_\a \to H$ denote the inclusion map. Of course   $i _\a^*: H \to H_\a$
 is essentially the same as $P_\a$:  more precisely $i _\a i _\a^*=P_\a$.
 We define $u_\a: M \to B(H_\a)$ by
 $u_\a(x)=  i _\a^* u(x)i _\a $   ($x\in M$), and 
 $v_{\a,\vp}: B(H_\a) \to M$ by 
 $v_{\a,\vp}(b) =W_{\a,\vp} (  i_\a b  i _\a^*) $ for any $b\in B(H_\a)$.
 Clearly 
 $$T_{\a,\vp}=W_{\a,\vp}u =v_{\a,\vp} u_\a.$$
   Therefore the maps in the net $(T_{\a,\vp})$ have the factorization required  in (ii) in Theorem \ref{ot1}. 
  Thus to conclude it suffices to verify that for any $y\in M$ 
   $$   \text{weak*}\lim  T_{\a,\vp}(y)= \text{weak*}\lim W_{\a,\vp}   u (y)=y$$
   or equivalently that  $$\text{weak*}\lim JW_{\a,\vp}   u (y)J=JyJ.$$
  Since the set $L$ is  total in $\ovl{M_*}$,   it suffices to show
  that $\lim \phi_x (\ovl{W_{\a,\vp}  u (y)}) = \phi_x(\ovl y)$ for any $x,y\in M$.
  We have (since $\|\xi_{\a,\vp} -\xi \|\to 0$) by \eqref{d32} and \eqref{d33} 
  $$\lim \phi_x(\ovl{W_{\a,\vp}   u (y)})= \lim\langle \xi_{\a,\vp},   J W_{\a,\vp}   u (y)J x \xi_{\a,\vp}\rangle =\lim \theta_{\a,\vp} (\ovl{  u (y)} ) (x)=\lim \theta_{\a } (\ovl{  u (y)} ) (x)$$
  $$=\lim   \tr (   u (y)^* h_\a   u(x)  h_\a ) =
   \langle \xi, Jy Jx \xi\rangle =\phi_x(\ovl y)
  .$$
  This completes the proof of (i) $\Rightarrow$ (ii).

Assume (ii) then the weak* PAP  is clear. In addition (ii)
implies that $Id_M$ factors as
$$Id_M=vu: M{\buildrel u\over\longrightarrow}   \ell_\infty(\{M_{n(\a)}\}) {\buildrel v\over\longrightarrow} M$$ 
where $v,u$ are unital positive contractions (and $u$ is normal).
Indeed, 
when $x\in M$ and $(a_\a)\in     \ell_\infty(\{M_{n(\a)}\})$, we may take $u(x)=(u_\a(x))$ and $v((a_\a))={\rm weak}^*\lim_\cl U v_\a(a_\a)$ for some suitable ultrafilter $\cl U$.
It follows that $M$ is isometric to a quotient Banach space of some $B(H)$
and hence by a result of Kirchberg (see \cite[Th. 15.5]{P6})
 $M$ is QWEP. Thus (ii) $\Rightarrow$ (iii).

Assume (iii). 
Let $(T_\a)_{\a\in I}$ be a net of  normal finite rank unital maps
 tending pointwise weak* to $Id_M$.
 Let $T_\a=v_\a u_\a$ be a factorization through $B(H_\a)$ as in 
 Lemma \ref{22/10}. Let $x\in M$ and $(a_\a)\in     \ell_\infty(\{B(H_\a)\})$.
 We define  $u: M \to \ell_\infty(\{B(H_\a)\})$   by $u(x)= (u_\a(x))_{\a\in I}$
 and $v((a_\a))={\rm weak}^*\lim_\cl U v_\a(a_\a)$ for some suitable ultrafilter $\cl U$.
This gives us a factorization of $Id_M$ of the form\\
$  M {\buildrel u\over\longrightarrow} \ell_\infty(\{B(H_\a)\}) {\buildrel v\over\longrightarrow} M$, with $u,v$ as required.
 Clearly $u^*(\ell_1(\{B(H_\a)_*\})   \subset M_*$, and hence $u$ is normal.
  Since we may
  obviously  replace $\ell_\infty(\{B(H_\a)\})$ by $B(H)$ with $H=\oplus H_\a$, this
  shows that $M$ is seemingly injective.\\
  For the last sentence in Theorem \ref{ot1}, 
suppose:
\item{(i)' } We have $Id_M=vu$  assuming merely that   $u$ and $v$ are positive maps,
with $u$ normal.  \\
Then, by Remark \ref{fac1}, $Id_M$ 
has a factorization as in \eqref{e1} with unital positive $u,v$ and $u$ normal.
A fortiori $\|u\|=\|v\|=1$.
As we already observed, the mere positivity of $v$ is enough to guarantee \eqref{282}.
But the latter was the only assumption needed on $v$ to complete the preceding argument for (i) $\Rightarrow$ (ii). This shows (i)' $\Rightarrow$ (ii)
and hence (i)' $\Rightarrow$ (i), and since the converse is trivial the proof is complete.
\end{proof}

\begin{rem}\label{or2} The preceding proof of (iii) $\Rightarrow$ (i) 
(resp. with $I$  countable) shows 
that any QWEP $M$  with the weak* BAP  is isomorphic to a complemented subspace
of $B(H)$ (resp. $B(\ell_2)$). More precisely, in the weak* BAP case, Lemma \ref{22/10}  gives us a factorization as in  \eqref{e1} 
 but in which the map $u$  is merely bounded and normal.  \end{rem}

We end this section with a characterization of 
the weak* PAP. It is not clear that    it is really different
from that of Theorem \ref{ot1}, because we do not know whether
the weak* PAP implies QWEP, as is the case in the c.p. case
for injectivity.

\begin{thm}\label{ot2} Consider a $\sigma$-finite von Neumann algebra  $M\subset B(H)$ in standard form. Let $\phi$ be a faithful normal state on $M$ and let
  $s $ be the associated self-polar form on $M$. The following  are equivalent:
\item{\rm (i)}  $M$ has the weak* PAP.
\item{\rm (ii)} There is a net of bipositive, separately normal
finite rank forms $s_\a$ tending pointwise to $s$.
\end{thm}
\begin{proof} 
Assume (i). Let $(T_\a)$ be  a net of  unital positive finite rank normal maps
tending to $Id_M$ pointwise weak*.
Let $s_\a(y,x)=s( T_\a(y), T_\a(x))$.
Since $s$ is separately weak* continuous,
we have $s(x,x) \le \liminf s_\a(x,x)$.
(Indeed, we have
by Cauchy-Schwarz
 $ |s( x, x)|^2=\lim |s( x, T_\a(x))|^2 \le \liminf |s( x, x)||s( T_\a(x), T_\a(x))|.$)
 Thus if $\cl U$ is an ultrafilter refining the net we have
$s(x,x)\le s'(x,x)$ with $s'(x,x)=\lim\nl_\cl U s_\a(x,x)$.
Since $s'$ is a bipositive normalized sesquilinear form
we have $s'=s$ by Theorem \ref{wc} whence  (ii).\\
Assume (ii). We will repeat the proof of the 
implication (i) $\Rightarrow$ (iii)
in Theorem \ref{ot1}, so we just sketch the argument.
We replace the net $(s_\a)$ by
$s_{\a,\vp}(y,x)= (1+\vp)^{-1}(s_\a(y,x) + \vp\ovl{\phi(y)} \phi(x))$.
Let $\phi_{\a,\vp}(x)=s_{\a,\vp}(1,x)$.
We define the unit vector $\xi_{\a,\vp} \in P^\natural$ by
$\phi_{\a,\vp}(x)=\langle \xi_{\a,\vp}, x \xi_{\a,\vp}\rangle$.
Then, using Corollary \ref{rc2}, we define the unital positive
(and hence contractive)  finite rank map $V_{\a,\vp} : M \to   M $ 
by
$$\forall x,y\in M\quad
s_{\a,\vp}(y,x)= \langle \xi_{\a,\vp}, J V_{\a,\vp}(y)  J x \xi_\a\rangle.$$
We have $\phi_{\a,\vp}  \to \phi $ weakly in $M_*$. Passing to convex combinations,
we may assume that  $\phi_{\a,\vp}  \to \phi $ in norm in $M_*$.
By \cite[Lemma 2.10]{Hasta} we have
$\xi_{\a,\vp}  \to \xi $ in norm in $H$. This allows us to pass from
$\lim  \langle \xi_{\a,\vp}, J V_{\a,\vp}(y)  J x \xi_{\a,\vp}\rangle=s(y,x)$
to 
$$\lim  \langle \xi, J V_{\a,\vp}(y)  J x \xi \rangle=s(y,x)
= \langle \xi, J y  J x \xi \rangle
.$$
By \eqref{f8} again we conclude that $V_{\a,\vp}(y) \to y$
for the weak* topology of $M$, which yields (i).
\end{proof}

\begin{rem}\label{cojo} In \cite{CoJo}, Connes and Jones studied an analogue of property (T) for
a von Neumann algebra $M$. When $M$ is equipped with a standard finite trace $\tau$  the latter property (T) is an obstruction 
to the existence
of certain approximations of the identity on $L_2(\tau)$
formed of compact \emph{c.p.} maps. 
More generally, in Theorem \ref{ot2} if $M$ has (T)
then (ii) cannot hold if  each bipositive form $s_\a$
defines a \emph{compact c.p.} map from $M$ to $\ovl{M_*}$.
We suspect that   some examples of $M$ with
property (T) can be found that enjoy a stronger property ruling out
the same for \emph{compact positive} maps; such examples would then fail the weak*    PAP.
\end{rem}

For completeness, we record here the following simple fact.
\begin{lem}\label{upmap} 
Let $M$ be 
  a von Neumann algebra.
  Assume that there is a uniformly bounded net of positive finite rank maps $(T_\a)$
  that tend pointwise weak* to $Id_M$. Then
$M$ has the   weak*    PAP.
\end{lem}
\begin{proof} We may assume $M$ $\sigma$-finite and in standard form with
a cyclic and separating unit vector $\xi\in P^\natural$.
For any $\vp>0$  we may replace $T_\a$ by
$x \mapsto T_\a(x) + \vp f(x) 1$ where $f$ is any fixed normal state,
so that 
 we may as well assume that $T_{\a }(1)$   is   invertible for all $\a$.
 As earlier let $\xi_{\a } \in P^\natural$ such that
$\langle \xi, J T_{\a }(1)J  x \xi \rangle= \langle \xi_{\a} ,    x  \xi_{\a} \rangle$.
By Corollary \ref{rc2} there is  a unital positive finite rank normal map
 $T'_{\a}: M \to M$  
such that $\langle \xi, J T_{\a }(y)J  x \xi \rangle= \langle \xi_{\a} ,  J T'_{\a}(y)J  x  \xi_{\a} \rangle$ for any $y\in M$. Since $T_{\a }(1) \to 1$ weak*,
the normal forms
$x \mapsto \langle \xi_{\a} ,    x  \xi_{\a} \rangle$ 
tend pointwise to $x \mapsto \langle \xi,    x \xi \rangle$ and are bounded in $M_*$.
Passing to convex combinations, we may assume the convergence is in
the norm of $M_*$. By \cite[Lemma 2.10]{Hasta} we have
$\xi_{\a }  \to \xi $ in norm in $H$. It follows that
$$\lim \langle \xi  ,  J T'_{\a}(y)J  x  \xi  \rangle =\lim\langle \xi_{\a} ,  J T'_{\a}(y)J  x  \xi_{\a} \rangle= \lim \langle \xi, J T_{\a }(y)J  x \xi \rangle= \langle \xi, J yJ  x \xi \rangle.$$
As before this implies that $T'_{\a}(y)\to y$ weak*.
\end{proof}

 \section{Comparison with injectivity}\label{inj}
 
 Our goal here is to highlight the analogy of our results
to  the previous ones on injectivity and semidiscreteness,
that came out of \cite{Co}.
 Connes's ideas \cite{Co} also led to the equivalence of injectivity with 
 approximate finite dimensionality (also called
 hyperfiniteness) and also with  the weak* CPAP
 (also called semidiscreteness).
 Since we have no analogue of approximate finite dimensionality (see however Remark \ref{hyp}),
 for which we refer to  chapter XVI in \cite{Tak2}  or to  \cite[chap. 11]{An2},
 we will concentrate on the parallel between   weak* CPAP$\Leftrightarrow$injective
 and weak* PAP$\Leftrightarrow$seemingly injective.
 We will focus on a result due to Effros and Lance from \cite[Prop. 4.5]{EL},
 with the refinement (i) $\Leftrightarrow$ (ii) proposed in \cite{[CE4]}.
\begin{thm}[\cite{EL}]\label{bla}  
The following properties of  a  von Neumann algebra $M \subset B(H)$ are equivalent:
\item{\rm (i)}
$M$  has the weak* CPAP (i.e. is semidiscrete according to \cite{EL}).
\item{\rm (ii)} There is a net of integers $n(\a)$ and normal  finite rank
   maps $T_\a: M\to M$ of the form
   $$    M {\buildrel u_\a\over\longrightarrow} M_{n(\a)} {\buildrel v_\a\over\longrightarrow} M$$
   such that $u_\a,v_\a$  are both unital and c.p.
    (so that $\|u_\a\|\le 1$,  $\|v_\a\|\le 1$), $u_\a$ is normal
   and $T_\a(x)=v_\a u_\a(x) \to x$ weak* for any $x\in M$.
\item[{\rm (iii)}] The
$*$-homomorphism $p: M' \otimes M\to B(H)$ defined by
 $p(y\otimes x)=yx$ extends to a contractive $*$-homomorphism
 on $M' \otimes_{\min} M $.
 \end{thm}
Since  the weak* CPAP does not depend
 on the embedding $M \subset B(H)$, one may choose
 a convenient embedding in (iii).
 If $M\subset B(H)$ is in standard form 
then $M'\simeq \ovl M$. In that case, $p$ becomes
the map taking $\ovl y\otimes x$ to $JyJx$, or equivalently
it becomes
the  representation defined on ${\ovl M \otimes_{\min} M}$ by
  $\rho(\sum \ovl{y_j} \otimes x_j )= \sum    J{y_j}J   x_j$ (see \S \ref{stan}  for the notation).
 Then the preceding statement says that $M$  has the weak* CPAP
if and only if 
for any   $t=\sum \ovl{y_j} \otimes x_j\in {\ovl M \otimes  M}$ we have
\begin{equation}\label{f2}  \| \rho(t)  \|_{B(H)} \le  \|t \|_{\ovl M \otimes_{\min} M} \end{equation}
Assume $M$ $\sigma$-finite for simplicity
and let 
  $\xi\in P^\natural$ be a unit vector that is both separating and cyclic.
  We claim that it suffices to have for any   $t\in {\ovl M \otimes  M}$    
\begin{equation}\label{f3}   |   \langle \xi,  \rho(t) \xi\rangle |\le  \|t \|_{\ovl M \otimes_{\min} M} .\end{equation}
Indeed, if this holds then $F$ defined by
$F(t) =  \langle \xi, \rho(t) \xi \rangle$ defines
a (vector) state on ${\ovl M \otimes_{\min} M}$,  and $\xi$ is a cyclic vector
for $\rho$. Thus the GNS representation of  $F$ on  ${\ovl M \otimes_{\min} M} $
(which is clearly contractive) can be identified with $\rho$.
So we recover \eqref{f2}. Thus  \eqref{f2} $\Leftrightarrow$ \eqref{f3}.

By a combination of results due to
 Haagerup (unpublished)  
for any finite set $(x_j)$ in $M$ we have  
\begin{equation}\label{f6} \| \sum \ovl{x_j} \otimes x_j \|_{\max} =
\|\rho(\sum    \ovl{x_j} \otimes x_j )\|_{B(H)}.\end{equation}
See \cite[Th. 23.39]{P6} for a detailed proof. In the semifinite case \eqref{f6} essentially goes back to \cite{Pro}.

The  identity \eqref{f6} 
shows that \eqref{f2} holds if and only if for any finite set $(x_j)$ in $M$ we have
\begin{equation}\label{ff7} \| \sum \ovl{x_j} \otimes x_j \|_{\ovl M \otimes_{\max} M} =
\| \sum \ovl{x_j} \otimes x_j \|_{\ovl M \otimes_{\min} M}.\end{equation}
This fact
 a posteriori
clarifies the equivalence between the weak* CPAP and  \eqref{f2},
which boils down to
(i) $\Leftrightarrow$ (iii) in Theorem \ref{bla}.
 This was
proved by Effros and Lance in \cite{EL} which circulated  (in preprint form) before
Connes and Choi-Effros  proved in \cite{Co,[CE2]} its equivalence with injectivity.
At the time Effros and Lance 
proved that the weak* CPAP  implies injectivity but
could not prove the converse. After the Connes paper \cite{Co}
it could be proved that injectivity implies 
(iii) in Theorem \ref{bla} (see \cite{Wac}), and that gave the desired converse.

\begin{thm}\label{n1} Let $M$ be a  von Neumann algebra.  The following are equivalent,
where in   (iii) we assume $M$ $\sigma$-finite in standard form,
and in (iii)' we assume  $M$ finite with a  faithful, normal and normalized trace $\tau$.
\item[{\rm (i)}] $M$ has the matricial weak* CPAP.
\item[{\rm (ii)}] There is a unital c.p. normal injective linear map $u: M \to B(H)$ such that $u^{-1}_{|u(M)}:u(M) \to M$ is positive 
and 
for any finite set $(x_j)$ in $M$ we have
$$\|\sum \ovl{x_j} \otimes x_j \|_{\ovl M \otimes_{\max} M} \le
   \|\sum \ovl{u(x_j)} \otimes u(x_j) \|_{\ovl{B(H)} \otimes_{\min}  {B(H)}} .$$
\item[{\rm (iii)}] There is a faithful normal state $\phi$ with associated
unit vector $\xi \in P^\natural$ such that
$$|\sum \langle \xi, Jx_j Jx_j \xi \rangle | \le \sup  \|\sum \ovl{u(x_j)} \otimes u(x_j) \|_{\ovl{B(H)} \otimes_{\min} B(H) } $$
where the sup runs over all $H$ and all  unital c.p. normal  maps $u: M \to B(H)$.
\item[{\rm (iii)'}]  For any finite set $(x_j)$ in $M$ we have
$$\sum \tau ( {x_j}^*   x_j)   \le   \sup\|\sum \ovl{u(x_j)} \otimes u(x_j) \|_{\ovl{B(H)} \otimes_{\min} B(H) } ,$$
where the sup runs over all $H$ and all    normal   $*$-homomorphisms $u: M \to B(H)$.
\end{thm}

This statement can be proved exactly in the same way as
for Theorem \ref{n2} below.

\begin{rem} By what precedes
 for $M$ finite equipped with $\tau$:\\
\eqref{f3}  holds if and only if 
for any finite set $(x_j)$ in $M$ we have
\begin{equation}\label{f5} \tau ( \sum {x_j}^*   x_j)   \le  \|\sum  \ovl{x_j} \otimes x_j \|_{\ovl M \otimes_{\min} M} .\end{equation}
We would like to indicate how this was originally
derived from Connes's ideas.
The usual representation by  left multiplication
of $M$ on $L_2(\tau)$ is a standard form for $M$. 
 Assuming $\tau(1)=1$,  we may take $\xi=1$.
 Then  \eqref{f3} holds
if and only if 
for any finite sets $(x_j)$, $(y_j)$ in $M$ we have
$$| \tau ( \sum {y_j}^*   x_j)  | \le  \|\sum \ovl{y_j} \otimes x_j \|_{\ovl M \otimes_{\min} M} .$$
In \cite{Co}, Connes observed that
if $M$ is a finite factor  then for \eqref{f3} to hold
it suffices that it holds for $x_j$ unitary and $x_j=y_j$.
The route he used went through the equivalence with injectivity.
The condition becomes $n\le \|\sum_1^n \ovl{x_j} \otimes x_j\|_{\min}$ for all finite sets 
of unitaries $(x_j)$, and now of course the converse inequality trivially holds. 
By a reasoning used commonly for  means on amenable groups, this
  leads to a net $(h_\a)$
of Hilbert-Schmidt operators with $\|h_\a\|_2^2=\tr |h_\a|^2=1$ such that
$\|h_\a -x h_\a x^*\|_2 \to 0$ for all unitary $x$.
Then Connes introduced the functional  $\Phi$ defined on $B(H)$ by $\Phi(T)=\lim_\cl U \tr(h_\a^* T h_\a)$ (that he called a ``hypertrace") from which
injectivity can be derived.
 By the same route, the restriction to factors was removed  in \cite[Lemma 2.2]{Ha}
by
Haagerup who proved that it suffices for \eqref{f3}
to hold for all $(x_j)$ of the form $(U_j q)$
with $(U_j)$ unitaries and $q\not=0$ an arbitrary central projection, as follows
\begin{equation}\label{f4}   n \le \|\sum  \ovl{U_j q} \otimes {U_j q} \|_{\ovl M \otimes_{\min} M} .\end{equation}
It is easy to deduce from this the following claim (for $M$ finite equipped with $\tau$):\\
\eqref{f3}   holds if and only if 
for any finite set $(x_j)$ in $M$ we have
\begin{equation}\label{f5} \tau ( \sum {x_j}^*   x_j)   \le  \|\sum  \ovl{x_j} \otimes x_j \|_{\ovl M \otimes_{\min} M} .\end{equation}
Indeed, if this holds then we must have
$ \tau (q)  n  \le  \|\sum  \ovl{U_j q} \otimes {U_j q} \|_{\ovl M \otimes_{\min} M} ,$
but then applying the same to  $(\sum  \ovl{U_j q} \otimes {U_j q})^m$
we  get 
$\tau (q)  n^m \le  \|\sum  \ovl{U_j q} \otimes {U_j q} \|^m_{\ovl M \otimes_{\min} M} ,$
and after taking the $m$-th root and letting $m\to \infty$ we recover \eqref{f4}. 
 \end{rem}

The analogue of the preceding statement for the  PAP 
in place of the CPAP is as  follows:

\begin{thm}\label{n2}  Let $M$ be a  von Neumann algebra.  
The following are equivalent,
where in   (iii) we assume $M$ $\sigma$-finite  in standard form,
and in (iii)' we assume  $M$ finite with  $\tau$ as before.
\item[{\rm (i)}] $M$ has the matricial weak* PAP.
\item[{\rm (ii)}] There is a unital positive normal injective linear map $u: M \to B(H)$ such that $u^{-1}_{|u(M)}:u(M) \to M$ is positive 
and 
for any finite set $(x_j)$ in $M$ we have
$$\|\sum \ovl{x_j} \otimes x_j \|_{\ovl M \otimes_{\max} M} \le   \|\sum \ovl{u(x_j)} \otimes u(x_j) \|_{\ovl{B(H)} \otimes_{\min} B(H)} .$$
 
\item[{\rm (iii)}] There is a faithful normal state $\phi$ with associated
unit vector $\xi \in P^\natural$   such that
$$|\sum \langle \xi, Jx_j Jx_j \xi \rangle | \le \sup  \|\sum \ovl{u(x_j)} \otimes u(x_j) \|
_{\ovl{B(H)} \otimes_{\min} B(H)} $$
where the sup runs over all $H$ and all  unital positive normal  maps $u: M \to B(H)$.
\item[{\rm (iii)'}] 
For any finite set $(x_j)$ in $M$ we have
$$\sum \tau ( {x_j}^*   x_j)   \le   \sup\|\sum \ovl{u(x_j)} \otimes u(x_j) \|_{\ovl{B(H)} \otimes_{\min} B(H) },$$
where the sup runs over all $H$ and all  unital positive normal  maps $u: M \to B(H)$.
\end{thm}



\begin{proof}
Assume (i).
By Theorem \ref{ot1}, $M$ is seemingly injective. Let $u,v$ be as
in \eqref{e1}. Then $vu(x_j)=x_j$ for any $x_j\in M$.
Therefore, since $v$ is here c.p.  \eqref{eoo3'} (with $u(x_j)$ in place of $x_j$)  implies (ii).
 (ii) $\Rightarrow$ (iii) 
   is clear by Remark \ref{cy2}.
 Assume (iii). 
 If we consider the single embedding $x\mapsto \oplus u(x)$ where
 the direct sum runs over all possible normal unital positive $u: M \to B(H)$
 (say with some bound on the cardinality of $H$)
 then we see that (iii) holds for a single $u$.
Furthermore, repeating that $u$ infinitely many times,
  we may assume that (iii) holds for some embedding $u$
  with infinite multiplicity.
 By the argument used in Proposition \ref{oo2} we have the situation described in
 \eqref{d30''}. The proof of Theorem \ref{ot1}
 shows that $M$ has the matricial weak*  PAP, and hence (i) holds.
In the finite case, the proof of Theorem \ref{ot1} shows that (iii)' $\Rightarrow$ (i) and (ii)
$\Rightarrow$ (iii)'  is clear by Remark \ref{cy2}.
 \end{proof}
 \begin{rem}\label{cou4}
 In sharp contrast with the c.p. case,
 we do not know whether the weak* PAP
 implies the matricial weak* PAP, or equivalently
 whether the weak* PAP implies QWEP.
\end{rem}

\section{Counterexamples}\label{cou}

In \cite{Sz} Szankowski proved that $B(H)$ fails the AP in Grothendieck's sense.
A fortiori, $B(H)^*$ fails MAP and hence $B(H)^{**}$
certainly fails the weak* PAP and hence is not seemingly injective.
Similarly $\bb B=(\oplus\sum M_n)_\infty$ is not seemingly injective.

We can derive from this the existence 
of finite factors (unfortunately not really too  ``concrete") that are not seemingly injective:

\begin{pro}\label{cou1} Let $\omega$ be a free ultrafilter on $\NN$.
The ultraproduct $\prod M_n/\omega$ and the ultrapower $R^\omega$
of the hyperfinite factor $R$ are not seemingly injective.
\end{pro}
 \begin{rem}\label{sup1} Proposition \ref{fac}
  obviously remains valid if $\{B(H_\a)\}$ is replaced by a family
  $\{M_\a\}$  of seemingly injective von Neumann algebras.
  \end{rem}
    \begin{rem}\label{sup2}  This shows in particular that
    if $\cl M$ is seemingly injective then any von Neumann subalgebra
    $M\subset \cl M$ admitting a unital positive (surjective) projection $P: \cl M \to M$
    is also seemingly injective. Note that when $\cl M$ is finite  there is automatically 
    a conditional expectation and hence a u.c.p. projection $P: \cl M \to M$.
  \end{rem}
  \begin{proof}[Proof of Proposition \ref{cou1}]  
  Assume for contradiction  that if either $\cl M=R^\omega$ or $\cl M=\prod M_n/\omega$, the algebra $\cl M$ is    seemingly injective.
  We will show that
  this implies that  \emph{any} QWEP von Neumann algebra $M$ is seemingly injective.
  \\
    Any  QWEP (or ``Connes embeddable") finite  $M$ 
      with separable predual
  embeds in $\cl M$ and hence is seemingly injective by  Remark \ref{sup2}. 
  When $M$ is finite with  $M_*$   non separable, this remains true if $M$ is
  $\sigma$-finite, or equivalently admits a faithful normal tracial state.
  Indeed, we may then view $M$ as the directed union of a family of
  finitely generated subalgebras $\{M_\a\}$ (with conditional expectations
  $P_\a: M \to M_\a$). Each $M_\a$ has a separable predual.
  Then Remark \ref{sup1} (or Proposition \ref{fac}
  applied to  $Id_M$) shows that any $M$ which is QWEP, $\sigma$-finite
 and finite is seemingly injective.
  By Remark \ref{sup1} again, the same remains true 
  if finite is replaced by  semifinite (indeed, we can apply Proposition \ref{fac}
  with $T=Id_M$ and with $\{B(H_\a)\}$  replaced by a family
  $\{M_\a\}$  of  QWEP   \emph{finite} and  $\sigma$-finite algebras of the form $M_\a=p_\a M p_\a$ where $p_\a$ are suitably chosen \emph{finite} projections in $M$). By Takesaki's duality theorem (as described in \cite[Th.11.3]{P6}),
  any   $M$ embeds as a von Neuman subalgebra
  in some \emph{semifinite}
   $\cl M$ in such a way that there is a 
  c.p. contractive projection from $\cl M$ onto $M$, and moreover
  $\cl M$  embeds in $M$
  as a von Neumann subalgebra in such a way that there is a u.c.p. projection
  from $M$ to $\cl M$. 
If $M$ is QWEP  and $\sigma$-finite
so is $\cl M$.
  (The fact that $\cl M$ inherits QWEP from $M$
  follows e.g. from Remark \ref{kiki}.)  
  Thus our inital assumption implies that any  $\sigma$-finite
  (and QWEP)  $M$ is seemingly injective.
   Now let $M$ be arbitrary (and QWEP). By a classical structural theorem
   (see  \cite[Ch. III \S 1 Lemma 7] {Dix}  (p.224 in the French edition
and p. 291 in the English one)
   $M$ admits a decomposition as a direct sum 
\begin{equation}\label{stru}
M\simeq (\oplus\sum\nl_{i\in I} B({\cl H_i}) \bar\otimes  N_i )_\infty\end{equation}
where the $N_i$'s are   $\sigma$-finite (=countably decomposable),
 the ${\cl H_i}$'s are   Hilbert spaces and $\bar\otimes$ denotes the (von Neumann algebra sense) tensor product. Note that the $N_i$'s inherit QWEP from $M$.\\
We claim that, under our assumption, any   QWEP $M$
as in \eqref{stru}  is seemingly injective.
Using Remark \ref{sup1} as before we can reduce this
to the case when $I$
is a countable (or even  finite) set and   the ${\cl H_i}$'s are separable (or even finite dimensional).
In the latter case $M$ is   $\sigma$-finite, 
so the first part of the argument
completes the proof of the claim. \\
Since we know (by \cite{Sz}) that $M=B(H)^{**}$
contradicts this, our initial assumption does not hold.
  \end{proof}

\section{Remote injectivity}\label{sci}

We remind the reader that $M$ is called remotely injective
if we have a factorization as in \eqref{e1}
but where we only assume $u$ normal and isometric, and 
 $v$  completely contractive. 
Let us first observe that this implies that $M$ is QWEP.
In fact it suffices for this to assume that $u,v$ are both contractions. Indeed,
this implies that $ M^*$ embeds isometrically into $ B(H)^*$, and hence that $M_*$
is finitely representable in $B(H)_*$ and this implies by Kirchberg's results (see \cite[\S 15]{P6}) that $M$ is QWEP.

The next statement gives a nice sounding reformulation
   of remote injectivity in terms of maximal operator spaces.
   The latter were introduced by Blecher and Paulsen \cite{[BP1]} (see e.g.  \cite[\S 3]{P4} for more on this) .
   Given a Banach space $X$ the operator space $\max(X)$ is 
   characterized by the property that it is isometrically isomorphic to $X$
   and for any map $u: \max(X) \to B(H)$ we have $\|u\|_{cb}=\|u\|$.
   One way to produce a completely isometric realization of
   $\max(X)$ is like this: let $\cl C$ denote the collection of
   all contractions $v: X \to B(H_v)$ (with $H_v$ either of cardinality
   at most that of $X$ or simply finite dimensional) and then
   consider the embedding 
   $$J: X \to B( \oplus_{v\in \cl C} H_v) \text{  defined by  } x\mapsto \oplus _{v\in \cl C} v(x).$$
   Then $ \max(X) $ can be identified with $J(X)\subset B( \oplus_{v\in \cl C} H_v)$.
   
   \begin{pro} \label{p23/10} A von Neumann algebra $M$ is remotely injective
   if and only if the natural complete contraction (defined by the identity map on $M$)  
   $$\Phi: \max(M) \to M$$
   factors completely contractively through  $B(H)$ for some $H$.
   \end{pro}
   \begin{proof}  Assume there are  $u,v$    
 such that $\Phi= vu$ with
   $\|u: \max(M) \to B(H)\|_{cb}=\|u\|=1$ and $\|v\|_{cb}=1$.
   Since $M$ is a dual space (i.e. $M=(M_*)^*$)
 the o.s. $\max(M)$ is a dual o.s.
 (it is the dual of the o.s. $\min(M_*)$ in the sense of \cite{[BP1]}).
 It follows (see e.g. \cite[Lemma 1.4.7 p. 23]{BLM}) that there is a completely isometric embedding
 $U: \max(M)\to  B(\cl H)$  realizing $\max(M)$ as a weak*-closed subspace.
 Let $X=U(\max(M))$ be the range of $U$.
 Then   $M$ is isometric to $X$ and by Sakai's uniqueness of predual  theorem
$U$ is normal.
 By the injectivity of $B(H)$ there is a
 map $\tilde u:  B(\cl H)\to B(H)$ with $\|  \tilde u\|_{cb}=\|   u\|_{cb}$
 extending $u$.
 Now if we let $V=v\tilde u: B(\cl H) \to M$
 we obtain $Id_M=VU$ with $V,U$ satisfying the conditions
 required to make $M$ remotely injective.
 This proves the if part. 
   The converse is 
  obvious.
\end{proof}

Ozawa   connected in \cite{Ozth} and \cite{Ozllp}
   the lifting problems described
  in Remark \ref{pb} with maximal operator spaces
  in the sense of Blecher and Paulsen \cite{[BP1]}  and local reflexivity in the sense of Effros and Haagerup (see \cite{ER} or \cite[\S 18]{P4}). Within operator spaces, it is convenient  to introduce constants
  relative to these properties, as follows.
  
  \begin{dfn}An operator space
   (o.s. in short)  $X\subset B(H)$ is called $\lambda$-maximal
    if  any bounded linear map $u: X \to Y$
  into an arbitrary operator space $Y$
  is completely bounded and satisfies $\|u\|_{cb} \le \lambda \|u\|$.
  \end{dfn}
  
  For example for any Banach space $X$ the o.s. $\max(X)$ described
above  is $1$-maximal.
  
  \begin{dfn}
  An o.s. $X\subset B(H)$ is called $\lambda$-locally reflexive
  if for any finite dimensional (f.d. in short)
subspace $E\subset X^{**}$ there is a net of maps
$u_\a: E \to X$ with $\|u_\a\|_{cb} \le \lambda$ that tend
pointwise-weak* to the inclusion $E\subset X^{**}$.
\end{dfn}
  
 To make the link  with Ozawa's questions,
we will use the following simple   lemma.
Whether an analogous lemma
 is valid with $Y$ finite dimensional (possibly in some  variant involving factoring
 the inclusion $E \subset X$ through $Y$) is the central open question
 discussed by Oikhberg in \cite{[O3]}.
 It is closely linked to Ozawa's question whether all 
 maximal spaces are locally reflexive.

 \begin{lem}\label{os1} Let $E\subset X$ be a  f.d. (or merely separable) subspace of a maximal
o.s. There is a  separable maximal subspace $Y $,
  such that  $E\subset Y \subset X$.
\end{lem}
\begin{proof} This is stated as Lemma 3.4 in \cite{Ozllp}. For a detailed proof see
  Exercise 3.8 in \cite[p. 80]{P4}, with solution in  \cite[p. 429]{P4}.
The latter is based on Paulsen's description of
the unit ball of $M_n(X)$ when $X$ is maximal (see \cite{[Pa5]} or e.g. \cite[p. 72]{P6}).
This lemma can also be checked   by a routine duality argument.
\end{proof}

An o.s. $X$ has the OLLP if any complete contraction
$u: X \to C/\cl I$ into an arbitrary quotient $C^*$-algebra
is locally liftable in the following sense: for any f.d. subspace $E\subset X$ there 
is a complete contraction $u^E: E \to C$ that lifts $u_{|E}$.

Let us say that an o.s. $X$ is strongly $\lambda$-maximal if  $X$ is the union of an increasing net of f.d.
$\lambda$-maximal subspaces $E_i$  such that $X=\ovl{\cup E_i}$.
 In \cite{Ozth} and \cite{Ozllp}
Ozawa observed that  
any strongly $1$-maximal $X$ has the OLLP.
Motivated by this observation,    
he asked whether any $1$-maximal 
o.s. has the OLLP.  
By Lemma \ref{os1}, the question reduces immediately to the 
case when $X$ is separable.
We think that the answer is negative
and  that there are   $1$-maximal 
o.s. that are not strongly $\lambda$-maximal for any $1\le \lambda<\infty$.
 In fact the space $X=\max(B(\ell_2)^{**})$ is a natural candidate for a counterexample.
 This question is related to our main topic via the following.

 \begin{pro} Let $M$ be a QWEP von Neumann algebra. If  $\max(M)$  
 has the  OLLP then  $M$ is
 remotely injective.
   \end{pro}
   \begin{proof} 
   As explained in the proof of Proposition \ref{p23/10} we may assume that
   we have a normal isometric embeding $U: M \to B(\cl H)$ realizing
   $\max(M) $ as a weak* closed subspace of $B(\cl H)$.
   For simplicity we identify  $\max(M)  $ with $U( M )$. 
   Let $W,q,w_1,w_2$ be as in Lemma \ref{23/10}.
   Let $u=\Phi: \max(M) \to M$.
   Let $E\subset \max(M)$ be f.d.
   By the OLLP there is $u^E: E \to W$ with $\| u^E \|_{cb} \le 1$
   such that $qu^E=u_{|E}$.
   By the injectivity of $B(H)$ the map $w_1 u^E: E \to B(H)$ admits an extension
   $v^E : B(\cl H) \to B(H)$ with $\| v^E \|_{cb} \le 1$. Then the map
   $w^E= w_2 v^E: B(\cl H) \to M$ is such that
   $w^E(x)=x$ for any $x\in E$  and $\|w^E\|_{cb} \le 1$.
   Let $\cl U$ be an ultrafilter refining the net of finite dimensional subpaces
   $\{E\mid E\subset M\}$. We define $V: B(\cl H) \to M$ by
   $$\forall b\in B(\cl H)\quad V(b)=\text{weak*}\lim\nl_\cl U w^E(b).$$
   Then $V(b)\in M$, $V(x)=x$ for any $x\in M$, so that $Id_M=VU$ and $\|V\|_{cb} \le 1$.
     \end{proof}

Let us say that a surjective $*$-homomorphism $q: W \to M$ 
(here between $C^*$-algebras) admits
 contractive
(resp. bounded) liftings if for any 
separable Banach space $Y$ and any contractive map $u: Y \to M$
there is a  contractive
(resp. bounded)  map  $\hat u: Y \to W$ that lifts $u$, meaning that $q\hat u=u$.
As discussed already in Remark \ref{pb}
 there are no known examples
 of   surjective $*$-homomorphism $q: W \to M$ that fail this.
 This motivates the next statement.
 
 \begin{pro}\label{23/10'} Let $M$ be a QWEP von Neumann algebra.
 Let $W$ be a WEP $C^*$-algebra and let  $q:W \to M$ be a surjective $*$-homomorphism. 
 If $q$ admits contractive liftings
 then $M$ is 
 remotely injective.
 \end{pro}
 
  \begin{proof} 
   Let  $U: M \to B(\cl H)$  be as in the preceding proof.
  We again identify  $\max(M)  $ with $U( M )$. 
 By Lemma \ref{os1}, for any f.d. $E\subset \max(M)$ there 
 is a separable $Y\subset  \max(M)$
 such that $E\subset Y$
 that is a maximal o.s. i.e.  we have
 $\max(Y)=Y$. 
 Let $u: Y \to  M$ be the inclusion map of $Y$ into $M$. By our assumption,
  $u$ admits a lifting
 $\hat u: Y  \to W$ with $\|\hat u\|\le 1$.
 Then $\|\hat u\|_{cb}=1$ by maximality of $Y$.   By the injectivity
 of $B(H)$, the map $w_1\hat u$ admits an extension
 $w_3: B(\cl H)\to B(H)$ with $\| w_3 \|_{cb} =1$. Let $w^E= w_2w_3: B(\cl H) \to M$.
 Note $w^E(x)=u(x)=x$ for any $x\in Y$, and a fortiori for any $x\in E$.
At this point we can conclude as in the preceding proof.
     \end{proof}
     
   $$\xymatrix{&B(\cl H) \ar@/^/[drr]^{w_3} &&  &\\
  &\max(M)\ar[u]^{U} &W\ar[d]^{q}\ar[r]^{  w_1 \ \   }&  B(H) \ar[dl]^{w_2} &
    \\
&Y\ar[ur]^{\hat u} \ar[r]^{u}\ar[u] & M  
   &   }$$
     
The same argument shows that if $q$ admits bounded liftings
then the isometric embedding $M \subset B(\cl H)$ is complemented,
i.e. there is a bounded projection $P: B(\cl H)\to M$.

 \begin{rem} We failed  to prove an analogue of Proposition \ref{23/10'}
 for positive contractive liftings to conclude that $M$ is  seemingly injective. We suspect the description
 in  \cite{PTT} of the analogue of $\max(X)$ when $X$ is an operator system should be
 useful, but the analogue of the \emph{normal} embedding $U$ seems to be missing.
 \end{rem}
 \begin{rem}\label{rrr} Let $M$ be a
  QWEP $C^*$-algebra. Then there is a normal
  embedding $M^{**}\subset B(H)^{**} $  with a unital completely positive 
  projection from $ B(H)^{**} $ onto $M^{**}$ (see \cite[p. 207]{P6}), and
  the projection can be chosen normal by \cite[p. 148]{P6}. 
  If $M$ is a von Neumann algebra there is a normal (non-unital) embedding
  $M\subset M^{**}$ (this is \emph{not} the canonical one)
  with a unital completely positive normal
  projection from $ M^{**} $ onto $M $.
  Therefore we have a factorization 
   $$
      Id_M: M {\buildrel U\over\longrightarrow} B(H)^{**} {\buildrel V\over\longrightarrow} M$$
  where $U,V$ are normal unital c.p. maps. Thus, if $B(H)^{**} $ was
  remotely  (resp. seemingly) injective, the same would be  true for  any QWEP $M$.
 \end{rem}
 
 The next lemma is a well known duality argument (see e.g. \cite[p. 440]{P6}).
 We denote there by $S_1^n$ the o.s. dual of $M_n$ that is determined
 by the isometric identity $CB(S_1^n, E)= M_n(E)$ for any o.s. $E$ (see \cite{ER,P4,Pa2}).
 \begin{lem} [Hahn-Banach argument]\label{hb}
Let $X\subset B$ be a completely isometric inclusion   of o.s.
Let $T: X \to C^*$ be a map into the dual of another o.s. $C$.
The following are equivalent:\\
(i) There is an extension $\tilde T: B \to C^*$ 
with $\| \tilde T \|_{cb }\le \lambda$.\\
(ii) For any $n$, any $S\subset S_1^n$  and any $u: S \to X$
that admits an extension $\hat u: S_1^n \to B$ 
with $\|\hat u\|_{cb}=1$ there is a map $\tilde u: S_1^n \to C^*$
such that ${\tilde u}_{|S} =Tu$
with $  \|\tilde u\|_{cb}\le \lambda$.
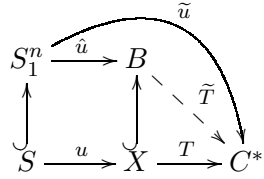
$$\xymatrix{S_1^n\ar[r]^{ \hat u }\ar@/^3pc/[drr]^{\tilde u}&B\ar@{-->}[dr]^{ \tilde T } \\
 S\ar@{^{(}->}[u]\ar[r]^{u} & X \ar@{^{(}->}[u] 
 \ar[r]^{ T \ } & C^*}$$
Note that if $B$ is injective (e.g. when $B=B(H)$) then (ii)
is the same as
\\
(ii)' For any $n$, any $S\subset S_1^n$  and any $u: S \to X$
with $\|  u\|_{cb}=1$ there is a map $\tilde u: S_1^n \to C^*$
such that ${\tilde u}_{|S} =Tu$
with $  \|\tilde u\|_{cb}\le \lambda$.
\end{lem}

\begin{thm} \label{t2}
Let $A$ be any unital $C^*$-algebra. 
If $\max(A)$ is locally reflexive (with constant 1) and $A$ WEP
then $A^{**}$ is remotely injective.
\end{thm}
\begin{proof} 
Let $U: \max(A^{**}) \to B(\cl H)$ be a completely isometric embedding as a weak*-closed subspace.
Let $\Phi: \max(A^{**})  \to  A^{**}$ be as before the complete contraction
defined by the identity on $A^{**}$. 
We claim that if $\max(A)$ is $\lambda$-locally reflexive
we have a factorization of   
$\Phi: \max(A^{**}) \to A^{**}$ of the form $\Phi=v U$ for some 
completely contractive $v:  B(\cl H) \to A^{**}$. To check that we use Lemma \ref{hb} and the local reflexivity. By (ii)' in Lemma \ref{hb}
it suffices  
to prove that for any $S\subset S_1^n $ and any $u: S \to \max(A^{**})$
with $\|u\|_{cb}=1$ there is a map $\tilde u: S_1^n \to A^{**}$
with $  \|\tilde u\|_{cb}\le \lambda$ that extends $u$ in the sense that $\tilde u_{|S}=\Phi u$.
Now since $\max(A)^{**}= \max(A^{**})$
the LR assumption on $\max(A)$ implies
that there is a net of c.c. maps $u_i: S \to \max(A)$ tending pointwise
weak* to $u$.
Now $\Phi u_i: S \to A$ is a complete contraction, and  since
$A$ has the WEP we have c.c. maps $v_i: S_1^n \to A$
extending $\Phi  u_i$. We now use an ultrafilter $\cl U$ and
define $\tilde u(x)=\lim_{\cl U} v_i(x) \in A^{**}$ for any $x\in S_1^n$, the limit being
in the $\sigma(A^{**},A^*)$-sense. Now $\tilde u: S_1^n \to A^{**}$
is our extension since for any $s\in S$ we have
$  \tilde u(s)=
\lim_{\cl U} v_i(s)=\lim_{\cl U} \Phi u_i(s)=\Phi u(s)$ for limits  in the same sense.
Thus $\tilde u_{|S}=\Phi u$.
\end{proof}

\begin{cor} 
If $\max(B(H))$ is 1-locally reflexive   then  $ B(H)^{**}$ is remotely injective.
\end{cor}
\begin{proof} 
Indeed, $B(H)$ is injective and hence has the WEP.
\end{proof}
 
\section{Some questions}\label{sq}
 
By a well known result (see \cite{[CE4], [CE4]})
a $C^*$-algebra is nuclear if and only if   $A^{**}$
is injective. One could name \emph{``seemingly nuclear"}
the $C^*$-algebras $A$ for which $A^{**}$ is seemingly injective.
Unfortunately at this stage we do not have any valuable information
on this notion.

Similarly, a discrete group $G$ could be called \emph{``seemingly amenable"}
if $M=L(G)$ is ``seemingly injective".  By Remark \ref{caha}  
this would lead to the shocking assertion that
   free groups
are seemingly amenable !

 At the moment we do not have any example of a group that is not 
 seemingly amenable but we suspect that
 such examples exist, for instance among   property (T)  groups
 (for which we refer the reader to \cite{BHV}).
 
 As for von Neumann algebras, it would be nice to exhibit more
 examples of $M$'s  that are QWEP but  not seemingly injective,
 without relying on Szankowski's construction \cite{Sz}.

Another puzzling question is whether it is really essential that the map
$u$ in Definition \ref{d1} should be \emph{normal}. It seems  essential for the proof
of Theorem \ref{ot1}. So  we ask: does any QWEP von Neumann algebra $M$
admit a factorization as in \eqref{e1} with $u$ isometric and $\|v\|_{cb}=1$ ?
 Is any non-nuclear QWEP von Neumann algebra $M$ isomorphic as a Banach space to
 $B(H)$ for some $H$ ?

Unfortunately, we  could not decide the following (which were our original motivation):\\
  {\bf Conjectures}: 
  \\(i) Any remotely injective von Neumann algebra $M$ has the (Banach space sense) weak* MAP  (equivalently $M_*$ has the MAP).
  \\(ii) $B(H)^{**}$ (say for $H=\ell_2$) is not remotely injective.
\\(ii)' There is a QWEP von Neumann algebra that  is not remotely injective.
 \\(iii) There exists
  an ideal in a separable
 $C^*$-algebra 
 with no contractive lifting,
 and hence with lifting constant $>1$.\\
 Note that by the same argument as for Proposition \ref{cou1} (ii) implies
 that $R^\omega$ or $\prod M_n/\omega$ are not remotely injective.
 We feel that (ii) and (iii) should be true,    but  are less convinced about (i), due to multiple unsuccessful attempts (our initial hope was to prove that remotely injective
 implies seemingly injective).
 
 We will show (i) $\Rightarrow$ (ii) $\Leftrightarrow$ (ii)' $\Rightarrow$ (iii).
 Obviously (i) implies (ii) by Szankowski's results on the failure of the AP
 for $B(H)$ or a fortiori for any non-nuclear von Neumann algebra. Indeed,
 his results imply that $B(H)^*$ fails the MAP, and hence that $B(H)^{**}$
 fails the  weak* MAP. (ii) $\Rightarrow$ (ii)' is trivial since
 $B(H)^{**}$ is QWEP (see \cite[p. 205]{P6}). In fact  (ii) and (ii)' are actually equivalent
 by Remark \ref{rrr}.\\
 The fact that (ii) implies (iii) follows from Proposition \ref{23/10'}.
   
 This argument shows that if (iii) fails then 
 any QWEP $M$ is remotely injective,
 or equivalently (ii) fails.
 
 In conclusion, if there exists a QWEP von Neumann algebra $M$ that is
 not remotely injective (which boils down to $M=B(H)^{**}$),
 then there exists
  an ideal in a separable
 $C^*$-algebra 
 with no contractive lifting, and 
 $\max(B(H))$ is not locally reflexive (with  constant 1). This would
 answer (at least  for the constant 1 case) some of the questions raised in \cite{Ozth} and \cite{Ozllp},
 and also in \cite{[O3]}.
 
 Lastly, we should mention that according to
 the recent  very long preprint \cite{JNVWY}
 there are von Neumann algebras that are not QWEP.
 
 \medskip
  
\medskip
 
  \medskip
   
 \noindent {\bf Acknowledgement: }Thanks are due to the referees for useful comments
  and corrections.


 \vfill\eject
 {\bf Corresponding author:}
 
 Gilles Pisier
 
 Mathematics Dept.
 
 Texas A\&M University
 
 College Station, TX 77843-3368, USA
 
 gilles.pisier@imj-prg.fr

 \bigskip
 {\bf Short statement on the research background and significance of the work:}
 
 This paper is devoted to a generalization of injectivity for von Neumann algebras. The surprising feature is that the
 generalized notion is satisfied by the  von Neumann algebra of a non commutative free
 group which is the main fundamental example
 of a non-injective von Neumann algebra. The background involves
 von Neumann algebra theory as well as 
 various notions of approximation properties of the identity by finite rank maps.
 
   \end{document}